\documentclass[final,a4paper,11pt]{article}
\usepackage{amssymb,amsmath,amsthm, amsfonts}
\usepackage{a4wide}
\usepackage{color}

\newcommand{\blue}[1]{\textcolor{blue}{#1}}

\newcommand{\at}[1]{}

\newcommand{\Gr}{{\rm gph\,}}

\newcommand{\co}{{\rm conv\,}}
\newcommand{\cl}{{\rm cl\,}}

\newcommand{\xb}{\bar x}
\newcommand{\yb}{\bar y}
\newcommand{\zb}{\bar z}
\newcommand{\wb}{\bar w}
\newcommand{\yba}{\yb^\ast{}}
\newcommand{\lb}{\bar\lambda}

\newcommand{\I}{{\cal I}}

\newcommand{\R}{\mathbb{R}}
\newcommand{\N}{\mathbb{N}}
\newcommand{\norm}[1]{\|#1\|}

\newcommand{\dist}[1]{{\rm d}(#1)}
\newcommand{\distb}[1]{{\rm d}\big(#1\big)}
\newcommand{\B}{{\cal B}}

\newcommand{\E}{{\cal E}}

\newcommand{\Sp}{{\cal S}}

\newcommand{\mv}{\,\vert\, }

\newcommand{\oo}{o}

\newcommand{\skalp}[1]{\langle #1\rangle}

\newcommand{\argmax}{\mathop{\rm arg\,max}\limits}

\newcommand{\vek}[1]{\left(\begin{array}{c}#1\end{array}\right)}

\newtheorem{definition}{Definition}
\newtheorem{theorem}{Theorem}
\newtheorem{lemma}{Lemma}
\newtheorem{example}{Example}
\newtheorem{corollary}{Corollary}
\newtheorem{proposition}{Proposition}

\begin{document}
\title{New constraint qualifications for
 mathematical programs with equilibrium
constraints via variational analysis}
\author{Helmut Gfrerer\thanks{Institute of Computational Mathematics, Johannes Kepler University Linz, A-4040 Linz, Austria, e-mail:
helmut.gfrerer@jku.at.} \and Jane J. Ye\thanks{Department of Mathematics
and Statistics, University of Victoria, Victoria, B.C., Canada V8W 2Y2, e-mail: janeye@uvic.ca.}}
\date{}
\maketitle
\begin{abstract}
In this paper, we study the mathematical program with equilibrium constraints (MPEC) formulated as a mathematical program with  a parametric generalized equation involving {the} regular normal cone. Compared with the usual way of formulating MPEC  through a KKT condition, this formulation has {the} advantage that it does not involve extra multipliers as new variables, and it usually requires weaker assumptions on the problem data. {U}sing the so-called first order sufficient condition for metric subregularity, we derive verifiable sufficient conditions for the metric subregularity of the involved set-valued mapping, or equivalently the calmness of the perturbed generalized equation mapping.

\vskip 10 true pt


\vskip 10 true pt


\end{abstract}

\section{Introduction}
A mathematical program with equilibrium constraints  (MPEC) usually refers to an optimization problem in which the essential constraints are defined by a parametric variational inequality or complementarity system.
Since many  equilibrium phenomena that arise from engineering and economics are  characterized by either  an optimization problem or a variational inequality, this justifies the name mathematical program with equilibrium constraints (\cite{Luo-Pang-Ralph, Out-Koc-Zowe}).
During the last two decades, more and more applications for MPECs have been found and there has been much progress made in both theories and algorithms for solving  MPECs.

For easy discussion, consider the following mathematical program with  a variational inequality constraint
\begin{eqnarray}
\mbox{(MPVIC)}\qquad\min_{(x,y)\in C} &&F(x,y)\nonumber\\
 \mbox{s.t.}&& \langle \phi(x,y), y'-y\rangle\geq 0 \quad \forall y'\in \Gamma, \label{VI}
\end{eqnarray}
where $C\subset \R^n\times\R^m$, $\Gamma:=\{y\in \R^m | g(y)\leq 0\}$,  $F:\R^n\times\R^m\to\R$, $\phi:\R^n\times \R^m\to\R^m$, $g:\R^m\to\R^q$ are sufficiently smooth. If the set $\Gamma$ is convex, then MPVIC can be equivalently written  as a mathematical program with a generalized equation constraint
 \begin{eqnarray*}
\mbox{(MPGE)}\qquad\min_{(x,y)\in C} &&F(x,y)\\
 \mbox{s.t.}&& 0\in \phi(x,y)+ {N}_\Gamma(y),
\end{eqnarray*}
where $N_\Gamma(y)$  is the normal cone to set $\Gamma$ at $y$ in the sense of convex analysis. If $g$ is affine or certain constraint qualification such as the Slater condition  holds for the constraint $g(y)\leq 0$, then it is known that
$  {N}_\Gamma(y)=\nabla g(y)^T N_{\R_-^q}(g(y)).$
Consequently
\begin{equation}\label{GE-equiv}
0\in \phi(x,y) +{N}_\Gamma(y) \Leftrightarrow \exists \lambda: 0\in \left (\phi(x,y)+\nabla g(y)^T\lambda, g(y)\right )+N_{\R^m \times \R_+^q}(y, \lambda),
\end{equation}
where $\lambda$ is referred to a multiplier.
This suggests to consider the mathematical program with a complementarity constraint
\begin{eqnarray*}
\mbox{(MPCC)}\qquad\min_{(x,y)\in C, \lambda\in \R^q} &&F(x,y)\\
 \mbox{s.t.}&& 0\in \left ( \phi(x,y)+\nabla g(y)^T\lambda, g(y)\right )+N_{\R^m \times \R_+^q}(y, \lambda).
\end{eqnarray*}
In the case where the equivalence (\ref{GE-equiv}) holds for  a unique multiplier $\lambda$ for each $y$, (MPGE) and (MPCC) are obviously equivalent while in the case where the multipliers are not unique then the two problems are not {necessarily} equivalent if the local optimal solutions are considered (see Dempe and Dutta \cite{Dam-Dut} in the context of bilevel programs). Precisely, it may be possible that for a local solution $(\bar x, \bar y, \bar \lambda)$ of (MPCC), the pair $(\bar x, \bar y)$ is not a local solution of (MPGE). This is a serious drawback for using {the} MPCC reformulation, since a numerical method computing a stationary point for (MPCC) may not have anything to do with the solution to the original MPEC. This shows that whenever possible, one should consider solving  problem (MPGE)  instead of  problem (MPCC). Another fact we want to mention  is that
in many equilibrium problems, the constraint set $\Gamma$ or the function $g$ may not be convex. In this case, if $y$ solves the variational inequality (\ref{VI}), then  $y'=y$ is a global minimizer of the optimization problem:
$\displaystyle \min_{y'} ~\langle \phi(x,y), y' \rangle \
 \mbox{ s.t. }  y'\in \Gamma,$
 and hence  by Fermat's rule (see, e.g., \cite[Theorem 10.1]{RoWe98})  it is a solution of the generalized equation
\begin{equation}\label{ge}
0\in \phi(x,y)+ \widehat{N}_\Gamma(y),\end{equation}
 where $\widehat{N}_\Gamma(y)$ is the regular normal cone to $\Gamma$ at $y$  (see Definition \ref{normalcone}). In the nonconvex case, by replacing the original variational inequality constraint (\ref{VI}) by the generalized equation (\ref{ge}), the feasible region is enlarged and the resulting MPGE may not be equivalent to the original MPVIC. However,  if the solution $(\xb,\yb)$ of MPGE is feasible for the original MPVIC, then it must be a solution of the original MPVIC; see \cite{Bouza} for this approach in the context of bilevel programs.

Based on the above discussion, in this paper we consider MPECs in the form
\begin{eqnarray*}
\mbox{(MPEC)}\qquad\min
 &&F(x,y)\\
 \mbox{s.t.}&& 0\in \phi(x,y)+ \widehat{N}_\Gamma(y),
 \\&&G(x,y)\leq 0,
\end{eqnarray*}
where $\Gamma$ is possibly non-convex and $G:\R^n\times\R^m\to\R^p$ is smooth.

Besides of the issue of equivalent problem formulations, one has to consider constraint qualifications as well. This task is of particular importance for deriving optimality conditions. For the problem (MPCC) there exist a lot of constraint qualifications from the MPEC-literature ensuring the Mordukhovich (M-)stationarity of locally optimal solutions. The weakest one of these constraint qualifications appears to be MPEC-GCQ (Guignard constraint qualification) as introduced by Flegel and Kanzow \cite{FleKan05}, see \cite{FleKanOut07} for a proof of  M-stationarity of local optimally solutions under MPEC-GCQ. For the problem (MPEC) it was shown by Ye and Ye \cite{YeYe} that calmness of the perturbation mapping associated with the constraints of (MPEC) (called pseudo upper-Lipschitz continuity in \cite{YeYe}) guarantees M-stationarity of solutions. \cite{Ad-Hen-Out} has compared the two  formulations (MPEC) and (MPCC)  in terms of calmness. {The authors} pointed out {there} that, very often, the calmness condition related to (MPEC) is satisfied at some $(\bar x, \bar y)$ while the one for (MPCC) are not fulfilled at $(\bar x,\bar y,\lambda)$ for certain multiplier $\lambda$.  In particular \cite[Example 6]{Ad-Hen-Out}  shows that it may be possible that the constraint for (MPEC) satisfies the calmness condition at $(\bar x,\bar y, 0)$ while the one for corresponding (MPCC) does not satisfy the calmness condition at $(\bar x,\bar y, \lambda, 0)$ for any multiplier $\lambda$. In this paper we first show that if  multipliers are not unique then {the} MPEC Mangasarian-Fromovitz constraint qualification (MFCQ)  never holds for problem (MPCC). Then  we present an example for which MPEC-GCQ is violated at $(\bar x,\bar y, \lambda, 0)$ for any multiplier $\lambda$ while the calmness holds for the corresponding (MPEC) at $(\bar x,\bar y,0)$. Note that  in contrast to
\cite[Example 6]{Ad-Hen-Out},  $\Gamma$ in our example is even convex.  However, very little is known how to verify the calmness for (MPEC) when the multiplier $\lambda$ is not unique. When $\phi$, $g$ and $G$ are affine, calmness follows  simply by Robinson's result on polyhedral multifunctions \cite{Robinson81}. Another approach is to verify calmness by showing the stronger Aubin property (also called  pseudo Lipschitz continuity or Lipschitz-like property) via the so-called Mordukhovich criterion, cf. \cite{Mor}. However, the Mordukhovich criterion involves the limiting coderivate of $\widehat N_\Gamma(\cdot)$, which is very difficult to compute in the case of nonunique $\lambda$; see \cite{GfrOut14b}.

 The main goal of this paper is to derive a simply verifiable criterion for the so-called metric subregularity constraint qualification (MSCQ); see Definition \ref{DefMetrSubregCQ}, which is equivalent to calmness. Our sufficient condition for MSCQ involves only first-order derivatives of $\phi$ and $G$ and derivatives up to the second-order of $g$ at $(\xb,\yb)$  and is therefore efficiently checkable. Our approach is mainly based on the sufficient conditions for metric subregularity as recently developed in \cite{Gfr11,Gfr13a,Gfr14b,GfrKl16} and some implications of metric subregularity which can be found in \cite{GfrMo16,GfrOut15}. A special feature is that the imposed constraint qualification on both the lower level system $g(y)\leq0$ and the upper level system $G(x,y)\leq0$ is only MSCQ, which is much weaker than MFCQ usually required.

\if{\blue{\noindent Recently a few papers have been devoted to MPECs in the  form of (MPEC).  \cite{HenOutSur} computed the regular coderivative for the solution mapping of the generalized equation (\ref{ge}) and used it to derive an S-stationary condition. \cite{Ad-Hen-Out} has compared the two  formulations (MPEC) and (MPCC)  in terms of calmness of the perturbation mapping associated with the generalized equations  (also called metric subregularity constraint qualification (MSCQ) in this paper; see Definition \ref{DefMetrSubregCQ}).  They pointed out that, very often, the calmness qualification condition related to (MPEC) is satisfied while the one for (MPCC) for certain multipliers $\lambda$ are not fulfilled. In particular \cite[Example 6]{Ad-Hen-Out}  shows that it may be possible that the constraint for the   MPEC satisfies the calmness condition at $(\bar x,\bar y, 0)$ while the one for corresponding MPCC does not satisfy the calmness condition at $(\bar x,\bar y, \lambda, 0)$ for any $\lambda\in \Lambda(\bar x, \bar y)$.
Recently \cite[Theorems 5 and 6]{GfrOut14} derived the formula for  the graphical derivative and the regular coderivative of the solution map to (MPEC) and consequently the strong (S-) stationary conditions for (MPEC)  under MSCQ and some  constraint qualifications that are weaker than the one given in \cite{HenOutSur}, cf. \cite[Remark 1]{GfrOut14}.  In order to obtain an Mordukhovich (M-) stationary condition,
it was shown in  (\cite{YeYe}) that MSCQ  is a constraint qualification. But for the general problem in the form of (MPEC), very little is known how to  verify MSCQ in case when the multipliers are not unique, except the case when $g$ is affine. In this paper we fill this gap by providing a verifiable sufficient condition for MSCQ for nonlinear $g$ under very weak assumptions.
}}\fi

 We organize our paper as follows. Section 2 contains the preliminaries and preliminary results. In section 3 we discuss the difficulties involved in formulating MPECs as (MPCC). Section
4 gives new verifiable sufficient conditions for MSCQ.

 The following {notation} will be used throughout the paper. We denote by $\B_{\R^q}$ the closed unit ball in $\R^q$ while when 
no confusion arises we denote it by $ \B$. By $\B(\zb; r)$ we denote the closed ball centered at $\zb$ with radius $r$.  $\Sp_{\R^q}$ is the unit sphere in $\R^q$.  For a matrix
$A$, we denote by $A^T$ its transpose. The inner product of two vectors $x, y$ is denoted by
$x^T y$ or $\langle x,y\rangle$ and by $x\perp y$ we mean $\langle x, y\rangle =0$. Let $\Omega \subset \R^d$ and $z \in \R^d$, we denote by $\dist{z, \Omega}$ the distance from $z $ to set $\Omega$. The polar cone of a set $\Omega$ is
$\Omega^\circ=\{x|x^Tv\leq 0 \ \forall v\in \Omega\}$ and $\Omega^\perp$ denotes the orthogonal complement to $\Omega$. For a set $\Omega$, we denote by $\co \Omega$ and $\cl\Omega$ the convex hull
and the closure  of $\Omega$ respectively. For a differentiable mapping $P:\mathbb R^d\rightarrow \mathbb R^s$, we denote by $\nabla P(z)$ the Jacobian matrix of $P$ at $z$ if $s>1$ and the gradient vector if $s=1$. For a function $f:\R^d \rightarrow \R$, we denote by $\nabla^2 f(\bar z)$ the Hessian matrix of $f$ at $\bar z$.  Let $M:\R^d\rightrightarrows\R^s$ be an arbitrary set-valued mapping, we denote its graph by $ {\rm gph}M:=\{(z,w)| w\in M(z)\}.$ $o:\R_+\rightarrow \R$ denotes a function with the property that $o(\lambda)/\lambda\rightarrow 0$ when $\lambda\downarrow 0$.

\section{{Basic definitions} and preliminary results}
In this section we gather some preliminaries and preliminary results in variational analysis that will be needed in the paper. The reader may find more details in the monographs \cite{Clarke,Mor,RoWe98} and  in the papers we refer to.
\begin{definition}
Given a set
$\Omega\subset\mathbb R^d$ and a point $\bar z\in\Omega$,
the (Bouligand-Severi) {\em tangent/contingent cone} to $\Omega$
at $\bar z$ is a closed cone defined by
\begin{equation*}\label{normalcone}
T_\Omega(\bar z)
:=\Big\{u\in\mathbb R^d\Big|\;\exists\,t_k\downarrow
0,\;u_k\to u\;\mbox{ with }\;\bar z+t_k u_k\in\Omega ~\forall ~ k\}.
\end{equation*}
The (Fr\'{e}chet) {\em regular normal cone} and the (Mordukhovich) {\em limiting/basic normal cone} to $\Omega$ at $\bar
z\in\Omega$ are  defined by
\begin{eqnarray}
&& \widehat N_\Omega(\bar
z):=\Big\{v^\ast\in\R^d\Big|\;\limsup_{z\stackrel{\Omega}{\to}\bar
z}\frac{\skalp{ v^\ast,z-\bar z}}{\|z-\bar z\|}\le
0\Big\} \nonumber \\
\mbox{and }  &&
N_\Omega(\bar z):=\left \{z^\ast \mv \exists z_{k}\stackrel{\Omega}{\to}\zb \mbox{ and } z^\ast_k\rightarrow z^\ast \mbox{ such that } z^\ast_{k}\in \widehat{N}_{\Omega}(z_k) \  \forall k \right \}
\nonumber
\end{eqnarray}
respectively.

\end{definition}
\noindent {Note that $\widehat N_\Omega(\bar z)=(T_\Omega(\bar z))^\circ$ and w}hen  the set $\Omega$ is convex, the tangent/contingent cone and the regular/limiting normal cone reduce to
the classical tangent cone and normal cone of convex analysis
respectively.

It is easy to see that $u\in T_\Omega(\bar z)$ if and only if $\liminf_{t\downarrow 0} t^{-1}\dist{\bar z+tu, \Omega}=0$.
Recall that a set $\Omega$ is said to be geometrically  derivable at a point $\zb\in \Omega$ if the tangent cone coincides with the derivable cone at $\xb$, i.e., $u\in T_\Omega(\bar z)$ if and only if $\lim_{t\downarrow 0} t^{-1}\dist{\bar z+tu, \Omega}=0$; see e.g. \cite{RoWe98}.
From the definitions of various tangent cones, it is easy to see that if a set $\Omega$ is Clarke regular in the sense of \cite[Definition 2.4.6]{Clarke} then it must be geometrically derivable and the converse relation is in general false. The following proposition therefore improves the rule of tangents to product sets given in \cite[Proposition 6.41]{RoWe98}. The proof is omitted since it follows from the definitions of the tangent cone and 
derivability.
\begin{proposition}[Rule of Tangents to Product Sets]\label{productset}  Let $\Omega=\Omega_1\times \Omega_2$ with $\Omega_1\subset \R^{d_1}, \Omega_2\in C^{d_2}$ closed. Then at any $\zb=(\zb_1,\zb_2)$ with $\zb_1\in \Omega_1, \zb_2\in \Omega_2$, one has
$$T_\Omega(\zb )\subset T_{\Omega_1}(\zb_1 )\times T_{\Omega_2}(\zb_2 ).$$
Furthermore the  equality holds if at least one of  sets $\Omega_1, \Omega_2$  is  geometrically  derivable.
\end{proposition}
The following directional version of the limiting normal cone  was introduced in \cite{Gfr13a}.
\begin{definition}Given a set
$\Omega\subset\mathbb R^d$, a point $\bar{z}\in \Omega$
and  a direction $w\in \mathbb{R}^{d}$, the {\em limiting normal cone} to $\Omega$ in  direction $w$ at $\bar{z}$ is defined by
\[
N_{\Omega}(\bar{z}; w):=\left \{z^{*} | \exists t_{k}\downarrow 0, w_{k}\rightarrow w, z^{*}_{k}\rightarrow z^{*}: z^{*}_{k}\in \widehat{N}_{\Omega}(\bar{z}+ t_{k}w_{k}) \ \forall k \right \}.
\]
\end{definition}
By definition it is easy to see that $N_{\Omega}(\bar{z}; 0)=N_{\Omega}(\bar{z})$  and $N_{\Omega}(\bar{z}; u)=\emptyset$ if $u\not \in T_{\Omega}(\bar{z})$.
Further by \cite[Lemma 2.1]{Gfr14b}, if $\Omega$ is a union of finitely many closed convex sets, then  one  has the following relationship between the limiting normal cone and its directional version.
\begin{proposition}\cite[Lemma 2.1]{Gfr14b}\label{relationship}
Let $\Omega\subset \R^d$ be a union of finitely many closed convex sets, $\bar z\in \Omega, u\in \R^d$. Then
$N_{\Omega}(\bar{z}; u)\subset  N_{\Omega}(\bar{z})\cap \{u\}^\perp $
and the equality holds if
 the set $\Omega$ is convex and $u\in T_\Omega(\bar z)$.
\end{proposition}

Next we consider constraint qualifications for a constraint system of the form
\begin{equation}
  \label{EqGenOptProbl}
z\in\Omega:=\{z\mv P(z)\in D\},
\end{equation}
where  $P:\R^d\to\R^s$ and $D\subset\R^s$ is closed.
\begin{definition}[cf. \cite{FleKanOut07}]Let $\zb\in \Omega$ where  $\Omega$ is defined as in (\ref{EqGenOptProbl}) with  $P$  smooth, and $T_\Omega^{\rm lin}(\bar z)$ be the linearized cone of $\Omega$ at $\bar z$ defined by
\begin{equation}
T_\Omega^{\rm lin}(\zb)=\{w|\nabla P(\bar z) w\in T_D (P(\bar z))\}.
\label{lcone}
\end{equation}
We say that the {\em generalized Abadie constraint qualification} (GACQ)  and the {\em generalized Guignard constraint qualification} (GGCQ) hold at $\zb$, if
\[T_\Omega(\zb)=T_\Omega^{\rm lin}(\zb) \mbox{ and } {(T_\Omega(\bar z))^\circ}=(T_\Omega^{\rm lin}(\zb))^\circ\]
respectively.
\end{definition}
It is obvious  that GACQ implies GGCQ  which is considered as the weakest constraint qualification. In the case of a standard nonlinear program,  GACQ and GGCQ reduce to the standard definitions of Abadie and Guignard constraint qualification {respectively}. Under  GGCQ, any local optimal solution to
a disjunctive problem, i.e., an optimization problem where the constraint has the form (\ref{EqGenOptProbl}) with the set $D$ equal to a union of finitely many polyhedral convex sets,  must be M-stationary (see e.g. \cite[Theorem 7]{FleKanOut07}).

GACQ and GGCQ are weak constraint qualifications, but they are usually difficult to verify. Hence we are interested in constraint qualifications that  are effectively verifiable, and yet not too strong. The following notion of metric subregularity is the base of the constraint qualification which plays a central role in this paper.
\begin{definition}
  	 Let $M:\R^d\rightrightarrows\R^s$ be a set-valued mapping and let $(\zb,\wb)\in\Gr M$. We say that $M$ is {\em metrically subregular} at $(\zb,\wb)$ if there exist a neighborhood $W$ of $\zb$ and a positive number  $\kappa>0$ such that
  \begin{equation}\label{EqMetrSubReg}\dist{z,M^{-1}(\wb)}\leq\kappa\dist{\wb,M(z)}\ \; \forall z\in W.
  \end{equation}
\end{definition}
The metric subregularity property was introduced in \cite{Ioffe79} for single-valued maps under the terminology ``regularity at a point'' and the name of ``metric subregularity''  was suggested in \cite{DoRo04}.
Note that the metrical subregularity at $(\zb,0)\in\Gr M$ is also referred to the existence of a local error bound at $\zb$.
It is easy to see that $M$ is metrically subregular at  $(\zb,\wb)$ if and only if its inverse set-valued map $M^{-1}$ is  calm at $(\wb,\zb)\in \Gr M^{-1}$, i.e.,   there exist a neighborhood $W$ of $\zb$,  a neighborhood $V$ of $\wb$ and a positive number  $\kappa>0$ such that
  $$M^{-1}(w)\cap V\subset M^{-1}(\wb) +\kappa \|w-\wb\| \B \quad \forall z\in W. $$
  While the term for the calmness of a set-valued map was first coined in \cite{RoWe98}, it was introduced as the pseudo-upper Lipschitz continuity  in \cite{YeYe} taking into { account} that it is weaker than both the pseudo Lipschitz continuity of Aubin \cite{Aubin} and the upper Lipschitz continuity of Robinson \cite{Robinson75,Robinson76} .

More general constraints can be easily written in the form $P(z)\in D$.
For instance,
a set
$\Omega=\{z\mv
P_1(z)\in D_1, 0\in P_2(z)+Q(z)\}$
where $P_i:\R^{d}\to\R^{s_i}$, $i=1,2$ and $Q:\R^{d}\rightrightarrows\R^{s_2}$ is a set-valued map can also be written as
\[\Omega=\{z\mv P(z)\in D\}\;\mbox{ with }\; P(z):=\left(\begin{array}{c}P_1(z)\\(z,-P_2(z))\end{array}\right),\ D:=D_1\times
\Gr Q.\]
 We now show that for both representations of $\Omega$ the properties of metric subregularity
  for the {maps} describing the constraints are equivalent.
\begin{proposition}\label{PropEquGrSubReg}Let $P_i:\R^{d}\to\R^{s_i}$, $i=1,2$,  {$D_1\subset \R^{s_1}$} be closed and  $Q:\R^{d}\rightrightarrows\R^{s_2}$ be a set-valued map with a closed graph. Further assume that $P_1$ and $P_2$ are Lipschitz near $\zb$. Then the set-valued map
\[M_1(z):=\left(\begin{array}
  {c}P_1(z)-D_1\\
  P_2(z)+Q(z)
\end{array}\right)\]
is metrically subregular at $(\zb,(0,0))$ if and only if the set-valued map
\[M_2(z):=\left(\begin{array}
  {c}P_1(z)\\(z,-P_2(z))
\end{array}\right)-D_1\times\Gr Q\]
is metrically subregular at $(\zb,(0,0,0))$.
\end{proposition}
\begin{proof}
  Assume without loss of generality that the image space $\R^{s_1}\times\R^{s_2}$ of $M_1$ is equipped with the norm $\norm{(y_1,y_2)}=\norm{y_1}+\norm{y_2}$, whereas we use the norm $\norm{(y_1,z,y_2)}=\norm{y_1}+\norm{z}+\norm{y_2}$ for the image space $\R^{s_1}\times\R^d\times\R^{s_2}$ of $M_2$. If $M_2$ is metrically subregular at $(\zb,(0,0,0))$, then there are a neighborhood  $W$ of $\zb$ and a constant $\kappa$ such that for all $z\in W$ we have
  \begin{eqnarray*}\dist{z,\Omega}&\leq& \kappa \distb{(0,0,0), M_2(z)}\\
  &=&\kappa\big(\dist{P_1(z),D_1}+\inf\{\norm{z-\tilde z}+\norm{-P_2(z)-\tilde y}\mv (\tilde z,\tilde y)\in\Gr Q\}\big)\\
  &\leq &  \kappa\big(\dist{P_1(z),D_1}+\inf\{\norm{-P_2(z)-\tilde y}\mv \tilde y\in Q(z)\}\big)=\kappa \distb{(0,0),M_1(z)} ,
  \end{eqnarray*}
  which shows metric subregularity of $M_1$. Now assume that $M_1$ is metrically subregular at $(\zb,(0,0))$ and hence we can find a radius $r>0$  and a real $\kappa$ such that
  \[\dist{z,\Omega}\leq\kappa \distb{(0,0),M_1(z)}\ \forall z\in \B(\zb;r).\]
  Further assume that {$P_1, P_2$ are} Lipschitz with modulus $L$ on $\B(\zb;r)$,  and consider $z\in \B(\zb;r/(2+L))$. Since $\Gr Q$ is closed, there are $(\tilde z,\tilde y)\in \Gr Q$ with
  \[\norm{z-\tilde z}+\norm{-P_2(z)-\tilde y}={\distb{(z,-P_2(z)),\Gr Q}}.\]
  Then
  \[\norm{z-\tilde z}\leq \distb{(z,-P_2(z)),\Gr Q}\leq \norm{z-\zb}+\norm{-P_2(z)+P_2(\zb)}\leq (1+L)\norm{z-\zb}\]
  implying $\norm{\zb-\tilde z}\leq \norm{\zb-z}+\norm{z-\tilde z}\leq (2+L)\norm{z-\zb}\leq r$ and
  \begin{eqnarray*}\dist{\tilde z,\Omega}&\leq& \kappa \distb{(0,0),M_1(\tilde z)}
  =\kappa\Big(\dist{P_1(\tilde z),D_1}+\distb{-P_2(\tilde z),Q(\tilde z)}\Big)\\
  &\leq& \kappa\big(\dist{P_1(\tilde z),D_1}+\norm{-P_2(\tilde z)-\tilde y}\big)\\
  &\leq &\kappa\big(2L\norm{z-\tilde z}+\dist{P_1(z),D_1}+\norm{-P_2(z)-\tilde y}\big).
  \end{eqnarray*}
  Taking into account $\dist{z,\Omega}\leq\dist{\tilde z,\Omega}+\norm{z-\tilde z}$ we arrive at
  \begin{eqnarray*}\dist{z,\Omega}&\leq&\kappa\max\{2L+\frac 1\kappa,1\}\big(\dist{P_1(z),D_1}+\norm{z-\tilde z}+\norm{-P_2(z)-\tilde y}\big)\\
  &=&\kappa\max\{2L+\frac 1\kappa,1\}\distb{(0,0,0),M_2(z)},
  \end{eqnarray*}
  establishing metric subregularity of $M_2$ at {$(\zb,(0,0,0))$}.
\end{proof}

Since the metric subregularity of the set-valued map $M(z):=P(z)-D$ at $(\zb,0)$ implies GACQ holding at $\bar z$, see e.g., \cite[Proposition 1]{HenOut05}, it can serve as a constraint qualification.
{Following \cite[Definition 3.2]{GfrMo15a}, we define it as a constraint qualification below.}
\begin{definition}[\bf metric subregularity constraint qualification]\label{DefMetrSubregCQ} {Let  $P(\zb)\in D$.
We say that the {\sc metric subregularity constraint qualification (MSCQ)} holds at $\zb$ for the system $P(z)\in D$} if the set-valued map $M(z):=P(z)-D$ is metrically subregular at $(\zb,0)$, or equivalently the perturbed set-valued map $M^{-1}(w):=\{z| w\in P(z)-D\}$ is calm at $(0, \zb)$.
\end{definition}
There exist several sufficient conditions for MSCQ in the literature. Here are the two most frequently used ones. The first case is  when the linear CQ holds, i.e., when $P$ is affine and $D$ is a union of finitely many polyhedral convex sets.
The second case is when the no nonzero abnormal multiplier constraint qualification (NNAMCQ) holds at $\zb$ (see e.g., \cite{Ye}):
\begin{equation}\label{NNAMCQ}
\nabla P(\zb)^T\lambda=0,\;\lambda\in N_D(P(\zb))\quad\Longrightarrow\quad\lambda=0.\end{equation}
 It is known that NNAMCQ is equivalent to MFCQ in the case of standard nonlinear programming. Condition (\ref{NNAMCQ}) appears under different terminologies in the literature; e.g., while it is called NNAMCQ in \cite{Ye}, it is referred to generalized MFCQ (GMFCQ) in \cite{FleKanOut07}.

The linear CQ and  NNAMCQ may be still too strong for some problems to hold.  Recently  some new constraint qualifications for standard nonlinear programs have been introduced in the literature that are stronger than MSCQ and weaker than the linear CQ and/or NNAMCQ; see e.g. \cite{AndHaeSchSilMP,AndHaeSchSilSIAM}. These CQs include the relaxed constant positive linear dependence condition (RCPLD) (see \cite[Theorem 4.2]{GuoZhangLin}), the constant  rank of the subspace component condition (CRSC) (see \cite[Corollary  4.1]{GuoZhangLin}) and the quasinormality \cite[Theorem 5.2]{guoyezhang-infinite}.

In this paper we will use  the following sufficient conditions.
\begin{theorem}\label{ThSuffCondMS}Let  $\zb \in \Omega$ where $\Omega $ is defined as in (\ref{EqGenOptProbl}).
MSCQ holds at $\zb$ if one of the following conditions is fulfilled:
\begin{itemize}
  \item First-order sufficient condition for metric subregularity (FOSCMS) for the system $P(z)\in D$ with $P$ smooth, cf. \cite[Corollary 1]{GfrKl16} : for every $0\not=w\in T_\Omega^{\rm lin}(\zb)$ one has
      \[\nabla P(\zb)^T\lambda=0,\;\lambda\in N_D(P(\zb);\nabla P(\zb)w)\quad\Longrightarrow\quad\lambda=0.\]
  \item Second-order sufficient condition for metric subregularity (SOSCMS) for the inequality  system $P(z)\in \R^s_-$ with $P$ twice
  Fr\'echet differentiable at $\zb$, cf. \cite[Theorem 6.1]{Gfr11}: For every $0\not=w\in T_\Omega^{\rm lin}(\zb)$ one has
      \[\nabla P(\zb)^T\lambda=0,\;\lambda\in N_{\R^s_-}(P(\zb)),\; w^T\nabla^2(\lambda^TP)(\zb)w\geq 0\quad\Longrightarrow\quad\lambda=0.\]
\end{itemize}
\end{theorem}
In the case $T_\Omega^{\rm lin}(\zb)=\{0\}$, FOSCMS satisfies automatically. By the definition of the linearized cone (\ref{lcone}), $T_\Omega^{\rm lin}(\zb)=\{0\}$ means that
 \[\nabla P(\zb)w=\xi, \quad \xi \in T_D(P(\zb))\;\Longrightarrow\;w=0.\]
 By the graphical derivative criterion for strong metric subregularity \cite[Theorem 4E.1]{DoRo14}, this is equivalent to saying that the set-valued map $M(z)=P(z)-D$ is strongly metrically subregular (or equivalently its inverse is isolated calm) at $(\bar z,0)$.
When the set $D$ is convex, by the relationship between the limiting normal cone and its directional version in Proposition \ref{relationship},
$$N_D(P(\zb);\nabla P(\zb)w)=N_D(P(\zb) )\cap \{\nabla P(\zb)w\}^\perp.$$
Consequently in the case where  $ T_\Omega^{\rm lin}(\zb ) \not =\{0\}$ and $D$ is convex, FOSCMS reduces to NNAMCQ. Indeed, suppose that $\nabla P(\zb)^T\lambda=0$ and $\lambda \in N_D(P(\zb) )$.  Then $\lambda^T (\nabla P(\zb)w)=0$. Hence $\lambda \in N_D(P(\zb);\nabla P(\zb)w)$ which implies from FOSCMS that $\lambda=0$.
Hence for convex $D$,
FOSCMS is equivalent to saying that either the strong metric subregularity or  the NNAMCQ (\ref{NNAMCQ}) holds at $(\bar z,0)$. In the case of an inequality system $P(z)\leq 0$ and $ T_\Omega^{\rm lin}(\zb ) \not =\{0\}$,
SOSCMS is obviously weaker than NNAMCQ.

In many situations, the constraint system $ P(z)\in D$ can be {splitted} into two parts such that one part can be easily verified to satisfy MSCQ. For example
\begin{equation}\label{EqConstrSplit}P(z)=(P_1(z),P_2(z))\in  D=D_1\times D_2\end{equation}
where $P_i:\R^d\to\R^{s_i}$ are smooth and $D_i\subset \R^{s_i}$, $i=1,2$ are closed, and for one part, let say $P_2(z)\in D_2$, it is known in advance that the map $P_2(\cdot)-D_2$ is metrically subregular at $(\zb,0)$. In this case the following theorem is useful.
\begin{theorem}
  \label{ThSuffCondMSSplit}{Let  $P(\zb)\in D$ with  $P$ smooth and $D$ closed
  and assume that $P$ and $D$} can be written in the form \eqref{EqConstrSplit} such that the set-valued map $P_2(z)-D_2$ is metrically subregular at $(\zb,0)$. Further assume for every $0\not=w\in T_\Omega^{\rm lin}(\zb)$ one has
     \[\nabla P_1(\zb)^T\lambda^1+\nabla P_2(\zb)^T\lambda^2=0,\;\lambda^i\in N_{D_i}(P_i(\zb);\nabla P_i(\zb)w)\;i=1,2\;\Longrightarrow\;\lambda^1=0.\]
      Then MSCQ holds at $\zb$ {for the system $P(z)\in D$}.
\end{theorem}
\begin{proof}
  Let the set-valued maps $M$, $M_i( i=1,2)$ be given by $M(z):=P(z)-D$ and $M_i(z)=P_i(z)-D_i( i=1,2)$ respectively. Since $P_1$ is assumed to be smooth, it is also Lipschitz near $\zb$ and thus $M_1$ has the Aubin property around $(\zb,0)$. Consider any direction $0\not=w\in T_\Omega^{\rm lin}(\zb)$. By \cite[Definition 2(3.)]{Gfr13a} the limit set critical for directional metric regularity ${\rm Cr}_{\R^{s_1}}M((\zb,0); w)$  {with respect to $w$ and $\R^{s_1}$ at $(\zb, 0)$} is defined as the collection of all elements $(v,z^\ast)\in\R^s\times\R^d$ such that there are sequences $t_k\searrow 0$, $(w_k,v_k,z_k^\ast)\to (w,v,z^\ast)$, $\lambda_k\in\Sp_{\R^s}$ and a real $\beta>0$ such that $(z_k^\ast,\lambda_k)\in\widehat N_{\Gr M}(\zb+t_kw_k, t_kv_k)$ and $\norm{\lambda_k^1}\geq\beta$ hold for all $k$, where $\lambda_k=(\lambda_k^1,\lambda_k^2)\in\R^{s_1}\times\R^{s_2}$. We claim that $(0,0)\not\in  {\rm Cr}_{\R^{s_1}}M((\zb,0); w)$. Assume on the contrary that $(0,0)\in  {\rm Cr}_{\R^{s_1}}M((\zb,0); w)$ and consider the corresponding sequences $(t_k,w_k,v_k,z_k^\ast,\lambda_k)$. The sequence $\lambda_k$ is bounded and by passing to a subsequence we can assume that $\lambda_k$ converges to some $\lambda=(\lambda^1,\lambda^2)$ satisfying $\norm{\lambda^1}\geq \beta>0$. Since $(z_k^\ast,\lambda_k)\in \widehat N_{\Gr M}(\zb+t_kw_k, t_kv_k)$ it follows from \cite[Exercise 6.7]{RoWe98} that
  $-\lambda_k\in\widehat N_D(P(\zb+t_kw_k)-t_kv_k)$ and $z_k^\ast=-\nabla P(\zb+t_kw_k)^T\lambda_k$ implying $-\lambda\in N_D(P(\zb);\nabla P(\zb)w)$ and $\nabla P(\zb)^T(-\lambda)=\nabla P_1(\zb)^T(-\lambda^1)+\nabla P_2(\zb)^T(-\lambda^2)=0$. From \cite[Lemma 1]{GfrKl16} we also conclude $-\lambda^i\in N_{D_i}(P_i(\zb);\nabla P_i(\zb)w)$ resulting in a contradiction to the assumption of the theorem. Hence our claim $(0,0)\not\in  {\rm Cr}_{\R^{s_1}}M((\zb,0); w)$ holds true and by \cite[Lemmas 2, 3, Theorem 6]{Gfr13a} it follows that $M$ is metrically subregular in direction $w$ at $(\zb,0)$, where directional metric subregularity is defined in \cite[Definition 1]{Gfr13a}. Since by {definition} $M$ is metrically subregular in every direction $w\not  \in T_\Omega^{\rm lin}(\zb)$, we conclude from \cite[Lemma 2.7]{Gfr14b} that $M$ is metrically subregular at $(\zb,0)$.
\end{proof}

We now discuss  some consequences of MSCQ. First we have the following change of coordinate formula for normal cones.
\begin{proposition}\label{PropInclNormalcone}
  Let $\zb\in \Omega:=\{z| P(z)\in D\}$ with $P$ smooth and $D$ closed.
  Then
  \begin{equation}\label{EqInclRegNormalCone}
  \widehat N_\Omega(\zb)\supset \nabla P(\zb)^T\widehat N_D(P(\zb)).
  \end{equation}
Further, if MSCQ holds at $\zb$ for the system $P(z)\in D$, then
  \begin{equation}\label{EqInclLimNormalCone}
  \widehat N_\Omega(\zb)\subset  N_\Omega(\zb)\subset \nabla P(\zb)^T N_D(P(\zb)).
  \end{equation}
   In particular if MSCQ holds at $\zb$ for the system $P(z)\in D$ with convex $D$, then
  \begin{equation} \label{EqInclNormalCone} \widehat N_\Omega(\zb)= N_\Omega(\zb)=\nabla P(\zb)^T N_D(P(\zb)).\end{equation}
\end{proposition}
\begin{proof}
  The inclusion \eqref{EqInclRegNormalCone} follows from \cite[Theorem 6.14]{RoWe98}.
  The first inclusion in {(\ref{EqInclLimNormalCone})} follows immediately from the definitions of the regular/limiting normal cone, whereas the second one  follows from \cite[Theorem 4.1]{HenJouOut02}. When $D$ is convex, the regular normal cone coincides with the limiting normal cone and hence (\ref{EqInclNormalCone}) follows by combining (\ref{EqInclRegNormalCone}) {and}  (\ref{EqInclLimNormalCone}).
\end{proof}

In the case where $D =\R^{s_1}_-\times\{0\}^{s_2}$, it is well-known in nonlinear programming theory that  MFCQ or equivalently NNAMCQ is a necessary and sufficient condition for the compactness of the set of Lagrange multipliers. In the case where $D\not =\R^{s_1}_-\times\{0\}^{s_2}$, NNAMCQ also implies the boundedness of the multipliers. However MSCQ is weaker than NNAMCQ and hence the set of Lagrange multipliers may be  unbounded if MSCQ holds but NNAMCQ fails. However Theorem \ref{ThBMP} shows that under MSCQ one can extract some uniformly compact subset of the multipliers.

\begin{definition}[cf. \cite{GfrMo16}]
  Let $\zb\in \Omega:=\{z| P(z)\in D\}$ with $P$ smooth and $D$ closed. We say that the {\em bounded multiplier property} (BMP) holds at $\zb$ for the system $P(z)\in D$, if there is some modulus $\kappa\geq 0$ and some neighborhood $W$ of $\zb$ such that for every $z\in W\cap \Omega$ and every $z^\ast\in N_{\Omega}(z)$ there is some $\lambda\in \kappa\norm{z^\ast}
  \B_{\R^s}\cap N_D(P(z))$ satisfying
  \[z^\ast=\nabla P(z)^T\lambda.\]
\end{definition}
The following theorem gives a sharper upper estimate for the normal cone  than (\ref{EqInclLimNormalCone}).
\begin{theorem}\label{ThBMP}  Let $\zb\in \Omega:=\{z\mv P(z)\in D\}$ and assume that MSCQ holds at the point $\zb$ for the system $P(z)\in D$.
Let $W$ denote an open neighborhood of $\zb$ and let $\kappa\geq 0$ be a real such that
$$\dist{z, \Omega} \leq \kappa \dist{P(z), D} \quad \forall z\in W.$$ Then
\[N_{\Omega}(z)\subset \left \{z^\ast\in\R^d\mv \exists \lambda\in \kappa\norm{z^\ast} \B_{\R^s}\cap N_D(P(z))\;\mbox{with}\; z^\ast=\nabla P(z)^T\lambda \right \}\quad \forall z\in W.\]
In particular BMP holds at $\zb$ for the system $P(z)\in D$.
\end{theorem}
\begin{proof} Under the assumption, the set-valued map $M(z):=P(z)-D$ is metrically subregular at $(\bar z,0)$. The definition of the metric subregularity justifies the existence of the open neighborhood $W$ and the number $\kappa$ in the assumption. Hence
   for each $z\in M^{-1}(0)\cap W=\Omega\cap W$ the map $M$ is also metrically subregular at $(z,0)$ and by applying  {\cite[Proposition 4.1]{GfrOut15}}  we obtain
    \[N_\Omega(z)=N_{M^{-1}(0)}(z;0)\subset\{z^\ast\mv \exists \lambda\in \kappa\norm{z^\ast}\B_{\R^s}: (z^\ast,\lambda)\in N_{\Gr M}((z,0);(0,0))\}.\]
It follows from \cite[Exercise 6.7]{RoWe98} that $$N_{\Gr M}((z,0);(0,0))=N_{\Gr M}((z,0))= \{(z^\ast,\lambda)\mv -\lambda\in N_D(P(z)), z^\ast=\nabla P(z)^T(-\lambda)\}.$$ Hence the assertion follows.
\end{proof}


\section{Failure of {MPCC-tailored  constraint qualifications for problem (MPCC)}}
In this section, we discuss difficulties involved in {MPCC-tailored} constraint qualifications for the problem (MPCC) by
considering the constraint system for problem
 (MPCC) in  the following form
$$
\widetilde \Omega :=\left \{(x,y,\lambda):  \begin{array}{l}
 0=h(x,y, \lambda):=\phi(x,y)+\nabla g(y)^T\lambda,\\
 0\geq g(y) \perp -\lambda \leq 0\\
 G(x,y)\leq 0\end{array}\right \},
$$
where    $\phi:\R^n\times \R^m\to\R^m$ and  $G:\R^n\times\R^m\to \R^p$ are   continuously differentiable and  $g:\R^m\to\R^q$ is twice continuously differentiable.

Given a triple $(\xb,\yb,\lb)\in \widetilde \Omega $ we define the following index sets of active constraints:
\begin{eqnarray*}
  &&{\cal I}_g:={\cal I}_g(\yb,\lb):=\{i{\in \{1,\ldots,q\}}\mv g_i(\yb)=0, \lb_i>0\},\\
  &&{\cal I}_\lambda:={\cal I}_\lambda(\yb,\lb):=\{i{\in \{1,\ldots,q\}}\mv g_i(\yb)<0, \lb_i=0\},\\
  &&{\cal I}_0:={\cal I}_0(\yb,\lb):=\{i{\in \{1,\ldots,q\}}\mv g_i(\yb)=0, \lb_i=0\},\\
  &&{\cal I}_G:={\cal I}_G(\xb,\yb):=\{i{\in \{1,\ldots,p\}}\mv G_i(\xb,\yb)=0\}.
\end{eqnarray*}
\begin{definition}[\cite{ScheelScholtes}]
We say that  MPCC-MFCQ holds at $(\xb,\yb,\lb)$ if the gradient vectors
\begin{equation}\label{MPEC-MFCQ} \nabla h_i(\xb, \yb, \lb), i=1,\dots, m, \  (0, \nabla g_i(\yb) ,0), i\in {\cal I}_g\cup {\cal I}_0, \  (0,0,e_i) , i\in {\cal I}_\lambda\cup {\cal I}_0,
\end{equation}
where  $e_i$ denotes the unit vector with the ith component equal  to $1$, are linearly independent and there exists a vector $(d_x,d_y,d_\lambda)\in\R^n\times\R^m\times\R^q$ orthogonal to the vectors in (\ref{MPEC-MFCQ}) and such that
$$\nabla G_i(\xb,\yb)(d_x,d_y)<0,  i\in {\cal I}_G.$$
We say that {MPCC-LICQ} holds at $(\xb,\yb,\lb)$ if the gradient vectors
$$\nabla h_i(\xb, \yb, \lb), i=1, \dots, m, \  (0,\nabla g_i(\yb) ,0), i\in {\cal I}_g\cup {\cal I}_0, \  (0,0,e_i)  , i\in {\cal I}_\lambda\cup I_0 ,\ (\nabla G_i (\bar x,\bar y),0), i\in {\cal I}_G
$$
are linearly independent.
\end{definition}
{MPCC-MFCQ} implies that for every partition $(\beta_1,\beta_2)$ of ${\cal I}_0$ the branch
\begin{equation}\label{branch}
\left \{ \begin{array}{l}
 \phi(x,y)+\nabla g(y)^T\lambda=0,\\
  g_i(y)=0, {\lambda_i\geq 0,  i\in {\cal I}_g,\; \lambda_i=0, g_i(y) \leq 0,}\; i\in {\cal I}_\lambda,\\
  g_i(y)=0, \lambda_i\geq 0,i\in\beta_1,\; g_i(y)\leq 0,\lambda_i=0,i\in\beta_2,\\
  G(x,y)\leq 0
  \end{array} \right.
\end{equation} satisfies MFCQ at $(\xb,\yb,\lb)$.

We now show that {MPCC}-MFCQ never {holds} for (MPCC)  if the lower level program has more than one multiplier.
\begin{proposition}\label{Prop5}
Let $(\xb,\yb,\lb)\in \widetilde \Omega$ and assume that there exists a second multiplier $\hat\lambda\not=\lb$ such that $(\xb,\yb,\hat \lambda)\in \widetilde \Omega$. Then for every partition $(\beta_1,\beta_2)$ of ${\cal I}_0$ the branch (\ref{branch})
does not fulfill MFCQ  at $(\xb,\yb,\lb)$.
\end{proposition}
\begin{proof}
  Since $\nabla g(\yb)^T(\hat\lambda-\lb)=0$, {$(\hat\lambda-\lb)_i\geq 0$, $i\in {{\cal I}_\lambda\cup \beta_2}$ and $\hat\lambda-\lb\not =0$,} the assertion follows immediately.
\end{proof}
Since {MPCC}-MFCQ is stronger than the {MPCC}-LICQ, we have the following corollary immediately.
\begin{corollary}\label{Cor1}Let $(\xb,\yb,\lb)\in \widetilde \Omega$ and assume that there exists a second multiplier $\hat\lambda\not=\lb$ such that $(\xb,\yb,\hat \lambda)\in \widetilde \Omega$. Then {MPCC}-LICQ fails at $(\xb,\yb,\lb)$.
\end{corollary}

{It is worth noting that our 
{result in Proposition \ref{Prop5} is} only valid under the assumption that $g(y)$ is independent of  $x$. In the case of bilevel programming where the 
lower level problem has a constraint dependent of the upper level variable, an example
given in \cite[Example 4.10]{Mehlitz-Wachsmuth} shows that  if the multiplier is not unique, then the corresponding MPCC-MFCQ may hold at some of the multipliers and fail to hold at others.}


\begin{definition}[see e.g. \cite{FleKanOut07}] Let $(\xb,\yb,\lb)$ be feasible for (MPCC). We say {MPCC-ACQ} 
 and 
 {MPCC-GCQ}  hold if
$$T_{\widetilde \Omega }(\xb,\yb,\lb)=T_{\rm MPCC}^{\rm lin}(\xb,\yb,\lb) \mbox{ and }  \widehat N_{\widetilde \Omega}(\xb,\yb,\lb)=(T_{\rm MPCC}^{\rm lin}(\xb,\yb,\lb))^\circ$$
respectively,  where
\begin{eqnarray*}\lefteqn{T_{\rm MPCC}^{\rm lin}(\xb,\yb,\lb)}\\
&:=&\left \{(u,v,\mu)\in\R^n\times\R^m\times\R^q\mv
\begin{array}{l}\nabla_{x} \phi(\xb,\yb)u+\nabla_y (\phi+\nabla_{y}(\lambda^Tg))(\xb,\yb)v +\nabla g(\yb)^T\mu=0,\\
\nabla g_i(\yb)v=0,i\in {\cal I}_g,\;\mu_i=0,i\in {\cal I}_\lambda,\\
\nabla g_i(\yb)v\leq 0,\mu_i\geq 0,\mu_i\nabla g_i(\yb)v=0, i\in {\cal I}_0,\\
\nabla G_i(\xb,\yb)(u,v)\leq 0,i\in {\cal I}_G
\end{array}\right \}\end{eqnarray*}
is the MPEC linearized cone at $(\xb,\yb,\lb)$.
\end{definition}
\noindent Note that {MPCC}-ACQ and {MPCC}-GCQ are the GACQ and GGCQ for the equivalent formulation of the set $\widetilde \Omega$ in the form of $P(z)\in D$ with $D$ involving the complementarity set
$$D_{cc}:=\{(a,b)\in \R_-^q\times \R_-^q| a^Tb=0\}$$
respectively. {MPCC}-MFCQ implies {MPCC}-ACQ (cf. \cite{FleKan05}) and from definition it is easy to see that {MPCC}-ACQ is stronger than {MPCC}-GCQ. Under {MPCC}-GCQ, it is known that a local optimal solution  of (MPCC) must be a M-stationary point (\cite[Theorem 14]{FleKanOut07}).  Although  {MPCC}-GCQ is weaker than most of other {MPCC-tailored} constraint qualifications,
the following example shows that the constraint qualification  {MPCC}-GCQ still can be violated  when the multiplier for the lower level  is not unique. In contrast to \cite[Example 6]{Ad-Hen-Out}, all the constraints are convex .
\begin{example}\label{Ex1}
Consider MPEC
\begin{eqnarray}
\min_{x,y} && F(x,y):=x_1-\frac{3}{2} y_1 + x_2-\frac 32 y_2- y_3 \nonumber \\
s.t. && 0\in \phi(x,y)+N_\Gamma( y), \label{EXGE}\\
&& G_1(x,y)=G_1(x):=-x_1-2x_2\leq 0,\nonumber \\
&& G_2(x,y)=G_2(x):=-2x_1-x_2\leq 0,\nonumber
\end{eqnarray}
where
$$\phi(x,y):=\left(\begin{array}{c}
  y_1-x_1\\
  y_2-x_2\\
  -1\end{array}\right), \quad \Gamma:=\left\{y\in \R^3|
g_1(y):=y_3+\frac12 y_1^2\leq 0,\;g_2(y):=y_3+\frac12 y_2^2\leq 0  \right\}.$$
Let $\xb=(0,0)$, $\yb=(0,0,0)$.
The lower level inequality system $g(y)\leq 0$ is convex  satisfying the Slater condition and therefore $y$ is a solution to the parametric generalized equation (\ref{EXGE})  if and only if $y'=y$ is a global minimizer of the optimization problem:
$\displaystyle \min_{y'} ~\langle \phi(x,y), y' \rangle \
 \mbox{ s.t. }  y'\in \Gamma,$ and if and only if there is a multiplier $\lambda$ fulfilling KKT-conditions
\begin{eqnarray}
\label{EqKKT}  &&\left(\begin{array}{c}
  y_1-x_1+\lambda_1y_1\\
  y_2-x_2+\lambda_2y_2\\
  -1+\lambda_1+\lambda_2\end{array}\right)=\left(\begin{array}
    {c}0\\0\\0
  \end{array}\right),\\
\nonumber && 0\geq y_3+\frac12 y_1^2\perp -\lambda_1\leq 0,\\
\nonumber &&0\geq y_3+\frac12 y_2^2\perp-\lambda_2\leq 0.
\end{eqnarray}
{Let ${\cal F}:=\{x\mv G_1(x)\leq 0, G_2(x)\leq 0\}$. Then  ${\cal F}={\cal F}_1\cup {\cal F}_2\cup{\cal F}_3$ where
\begin{eqnarray*}
  {\cal F}_1&:=&\left \{(x_1,x_2)\in \R^2 \mv 2\vert x_2\vert\leq x_1 \right \},\\
  {\cal F}_2&:=&\left\{(x_1,x_2)\in  \R^2\mv \frac {x_1}2\leq x_2\leq 2x_1\right\},\\
  {\cal F}_3&:=&\left\{(x_1,x_2)\in \R^2\mv  2\vert x_1\vert\leq x_2\right\}.
\end{eqnarray*}
Straightforward calculations yield that for each $x\in {\cal F}$ there exists a unique solution $y(x)$, which is given by
\[y(x)=\begin{cases}(\frac{x_1}2, x_2,-\frac 18 x_1^2)& \mbox{if $x\in{\cal F}_1$,}\\
( \frac{x_1+x_2}3, \frac{x_1+x_2}3,-\frac1{18}(x_1+x_2)^2)& \mbox{if $x\in{\cal F}_2$,}\\
(  x_1, \frac{x_2}2,-\frac 18 x_2^2)& \mbox{if $x\in{\cal F}_3$.}
\end{cases}\]}
Further, at {$\xb=(0,0)$ {we have} $y(\bar x)=(0,0,0)$} and  the set of the multipliers is
 $${ \Lambda:=}\{\lambda\in \R_+^2| \lambda_1+\lambda_2=1\},$$
while for all {$x\not=(0,0)$} the gradients of the lower level constraints active at $y(x)$ are linearly independent and the unique multiplier is given by
\begin{equation}\label{lambda}\lambda(x)=\begin{cases}(1,0)& \mbox{if $x\in{\cal F}_1$,}\\
( \frac{2x_1-x_2}{x_1+x_2}, \frac{2x_2-x_1}{x_1+x_2})&\mbox{if $x\in{\cal F}_2$,}\\
( 0,1)& \mbox{if $x\in{\cal F}_3$.}
\end{cases}\end{equation}
Since
\[F(x,y(x))=\begin{cases}
\frac 14 x_1-\frac 12 x_2+\frac 18 x_1^2&\mbox{if $x\in{\cal F}_1$},\\
\frac1{18}(x_1+x_2)^2&\mbox{if $x\in{\cal F}_2$},\\
\frac 14 x_2-\frac 12 x_1+\frac 18 x_2^2&\mbox{if $x\in{\cal F}_3$},\\
\end{cases}\]
{and ${\cal F}={\cal F}_1\cup {\cal F}_2\cup{\cal F}_3$,}  we see that $(\xb,\yb)$ is a globally optimal solution of the MPEC.

The original problem is equivalent to the following MPCC:
\begin{eqnarray*}
  \min_{x,y,\lambda}&& x_1-\frac 32 y_1 + x_2-\frac 32 y_2- y_3\\
  \mbox{s.t.}&&\mbox{$x,y,\lambda$ {fulfill \eqref{EqKKT},}}\\
  &&-2x_1-x_2\leq 0,\\
  &&-x_1-2x_2\leq 0.
\end{eqnarray*}
The feasible region  of this problem is
\begin{eqnarray*}
\widetilde \Omega=\bigcup_{\xb\not=x\in {\cal F}}\{(x,y(x),\lambda(x))\}
{\cup(\{(\xb,\yb)\}\times \Lambda)}.
\end{eqnarray*}
Any  $(\xb,\yb, \lambda)$  where $\lambda \in {\Lambda}$ is a globally optimal solution. However
it is easy to verify that unless  $\lambda_1=\lambda_2=0.5$,  any  $(\xb, \yb, \lambda)$ is not even a weak stationary point, implying by \cite[Theorem 7]{FleKanOut07} that {MPCC}-GCQ and consequently {MPCC}-ACQ {fails}  to hold. Now consider $\lambda=(0.5,0.5)$.
The MPEC  linearized cone $T^{\rm lin}_{\rm MPCC}(\xb,\yb,\lambda)$ is the collection of all $(u,v,\mu)$ such that
\begin{equation}\label{KKTeqn}
 \left(\begin{array}{c}
  1.5v_1-u_1\\
  1.5v_2-u_2\\
  \mu_1+\mu_2\end{array}\right)=\left(\begin{array}{c}
  0\\
  0\\
  0\end{array}\right),\quad
\begin{array}{l}v_3=0,\\
  -2u_1-u_2\leq 0,\; -u_1-2u_2\leq 0.\end{array}
\end{equation}
Next we compute the actual tangent cone $T_{\widetilde \Omega}(\xb,\yb,\lambda)$. Consider sequences $t_k\downarrow 0$, $(u^k,v^k,\mu^k)\to (u,v,\mu)$ such that $(\xb,\yb,\lambda)+t_k(u^k,v^k,\mu^k)\in\widetilde\Omega$. If  $u^k\not=0$ for infinitely many $k$, then {$\bar x+t_k u^k\not =0$} and hence $(\yb+t_kv^k,\lambda+t_k\mu^k)=(y(\xb+t_ku^k),\lambda(\xb+t_ku^k))$ for those $k$. Since $\lambda=(0.5,0.5)$, it follows from (\ref{lambda}) that $\xb+t_ku^k\in{\cal F}_2$ for infinitely many $k$, implying, by passing to a subsequence if necessary,
\[v=\lim_{k\to \infty}\frac{y(\xb+t_ku^k)-\yb}{t_k}=\frac 13(u_1+u_2,u_1+u_2,0)\]
and
\begin{eqnarray*}\mu&=&\lim_{k\to\infty}\frac{\lambda(\xb+t_ku^k)-\lambda}{t_k}=
\lim_{k\to\infty}\frac{(\frac{2u_1^k-u_2^k}{u_1^k+u_2^k}, \frac{2u_2^k-u_1^k}{u_1^k+u_2^k})-(0.5,0.5)}{t_k}\\
&=&\lim_{k\to\infty}1.5\frac{(\frac{u_1^k-u_2^k}{u_1^k+u_2^k}, \frac{u_2^k-u_1^k}{u_1^k+u_2^k})}{t_k}.\end{eqnarray*}
Hence $v_1=v_2=\frac{1}{3}(u_1+u_2),\ v_3=0$ {and $\mu_1+\mu_2=0$.}
{Also from (\ref{KKTeqn}), we have $u_1=u_2$ since $v_1=v_2$ and the tangent cone $T_{\widetilde \Omega}(\xb,\yb,\lambda)$ is  always a subset of the  MPEC  linearized cone $T^{\rm lin}_{\rm MPCC}(\xb,\yb,\lambda)$ (see e.g. \cite[Lemma 3.2]{FleKan05}).}
  Further, since  $\xb+t_ku^k\in{\cal F}_2$, we must have $u_1\geq 0$.
If $u^k=0$ for all but finitely many $k$, then we have $v^k=0$ and $\lambda+t_k \mu^k\in\Lambda$ implying $\mu_1+\mu_2=0$. Putting all together, we obtain that
 the actual tangent cone $T_{\widetilde\Omega}(\xb,\yb,\lambda)$ to the feasible set is the collection of all $(u,v,\mu)$ satisfying
 \begin{eqnarray*}
  &&u_1=u_2\geq 0, v_1=v_2=\frac 23 u_1,\\
  &&v_3=0, \mu_1+\mu_2=0.
 \end{eqnarray*}
 Now it is easy to see that $T_{\widetilde\Omega}(\xb,\yb,\lambda)\not=T^{\rm lin}_{\rm MPCC}(\xb,\yb,\lambda)$. Moreover since both $T_{\widetilde\Omega}(\xb,\yb,\lambda)$ and $T^{\rm lin}_{\rm MPCC}(\xb,\yb,\lambda)$ are convex polyhedral {sets}, one also has
$(T_{\widetilde\Omega}(\xb,\yb,\lambda))^\circ\not=(T^{\rm lin}_{\rm MPCC}(\xb,\yb,\lambda))^\circ$
and thus MPEC-GCQ does not hold  for $\lambda=(0.5,0.5)$ as well.
\if{We will now show that at any solution point $(\xb,\yb,\lambda)$ the constraint qualifications MPEC-ACQ and MPEC-GCQ fail to hold.
Consider $\lambda=(1,0)$.
The MPEC linearized cone $T^{\rm lin}_{\rm MPCC}(\xb,\yb,\lambda)$ is the collection of all $(u,v,\mu)$ such that
\begin{eqnarray*}
 \left(\begin{array}{c}
  2v_1-u_1\\
  v_2-u_2\\
  \mu_1+\mu_2\end{array}\right)=0,\quad
\begin{array}{l}v_3=0,\\
  v_3\leq 0,\; \mu_2\geq 0,\; \mu_2v_3=0,\\
  -2u_1-u_2\leq 0,\; -u_1-2u_2\leq 0,\end{array}
\end{eqnarray*}
which can be simplified to
\begin{eqnarray*}
 \left(\begin{array}{c}
  2v_1-u_1\\
  v_2-u_2\\
  \mu_1+\mu_2\end{array}\right)=0,\quad
\begin{array}{l}v_3=0,\\
\mu_2\geq 0,\\
  -2u_1-u_2\leq 0,\; -u_1-2u_2\leq 0.\end{array}
\end{eqnarray*}
Next we compute the actual tangent cone $T_\Omega(\xb,\yb,\lambda)$. Consider sequences $t_k\downarrow 0$, $(u^k,v^k,\mu^k)\to (u,v,\mu)$ such that $(\xb,\yb,\lambda)+t_k(u^k,v^k,\mu^k)\in\Omega$ and $u^k\not=0$ $\forall k$. If $\xb+t_ku^k\in{\cal F}_1$ for infinitely many $k$, then immediately $u_1\geq 2\vert u_2\vert$, {$v=(\frac {u_1}2,u_2, 0)$}, $\mu=(0,0)$ follows. On the other hand, if $\xb+t_ku^k\in{\cal F}_2$ for infinitely many $k$, then, by passing to a subsequence if necessary,
\begin{eqnarray*}\mu&=&\lim_{k\to\infty}\frac{\lambda(\xb+t_ku^k)-\lambda}{t_k}=
\lim_{k\to\infty}\frac{(\frac{2u_1^k-u_2^k}{u_1^k+u_2^k}, \frac{2u_2^k-u_1^k}{u_1^k+u_2^k})-(1,0)}{t_k}\\
&=&\lim_{k\to\infty}\frac{(\frac{u_1^k-2u_2^k}{u_1^k+u_2^k}, \frac{2u_2^k-u_1^k}{u_1^k+u_2^k})}{t_k}.\end{eqnarray*}
Hence $2u_2=u_1$ and $\mu_1+\mu_2=0$. Finally, the case $\xb+t_ku^k\in{\cal F}_3$ is not possible due to $\lambda(x)=(0,1)$ $\forall 0\not=x\in {\cal F}_3$. Putting all together, we obtain that
 the actual tangent cone to the feasible set is the collection of all $(u,v,\mu)$ satisfying
 \begin{eqnarray*}
 \left(\begin{array}{c}
  2v_1-u_1\\
  v_2-u_2\\
  \mu_1+\mu_2\end{array}\right)=0,\quad
\begin{array}{l}v_3=0,\\
 \mu_2\geq 0,\; \mu_2(2u_2-u_1)=0,\\
  u_1\geq 2\vert u_2\vert.\end{array}
\end{eqnarray*}
Since
$(T_\Omega(\xb,\yb,\lambda))^\circ\not=(T^{\rm lin}_{\rm MPCC}(\xb,\yb,\lambda))^\circ$,
 MPCC-GCQ does not hold. The other cases $\lambda_1,\lambda_2>0, \lambda_1+\lambda_2=1$ and $\lambda=(0,1)$ can be treated similarly and we again obtain that  MPCC-GCQ does not hold.
}\fi

\end{example}

\section{Sufficient condition for MSCQ}
As we discussed in the introduction and section 3, there are much difficulties involved in  formulating {an}  MPEC  as (MPCC). In this section, we turn our attention to  problem (MPEC)
with the constraint system defined in the following form
\begin{equation}\label{Omega}
\Omega:=\left \{(x,y):  \begin{array}{l}
0\in \phi(x,y)+ \widehat{N}_\Gamma(y)\\
G(x,y)\leq 0\end{array}\right \},
\end{equation}
where
$\Gamma:=\{y\in \R^m | g(y)\leq 0\}$,
$\phi:\R^n\times \R^m\to\R^m$ and  $G:\R^n\times\R^m\to \R^p$
are   continuously differentiable and  $g:\R^m\to\R^q$ is twice continuously differentiable.
Let $(\xb,\yb)$ be a feasible solution of problem (MPEC).
We assume that  MSCQ is fulfilled for the  constraint $g(y)\leq 0$ at $\yb$. Then by definition MSCQ also holds for all points $y\in\Gamma$ near $\yb$ and by Proposition \ref{PropInclNormalcone}  the following equations hold for such $y$:
\[N_\Gamma(y)=\widehat N_\Gamma(y)=\nabla g(y)^T N_{\R^q_-}(g(y)),\]
where $N_{\R^q_-}(g(y))=\{\lambda\in\R^q_+\mv\lambda_i=0, i\not\in \I(y)\}$ {and $\I(y):=\{i\in\{1,\ldots,q\}\mv g_i(y)=0\}$ is the index set of active inequality constraints.}

For the sake of simplicity we do not include equality constraints in either the upper or the lower level constraints. We are using MSCQ as the basic constraint qualification for both the upper and the lower level constraints and this allows us to write an equality constraint $h(x)=0$ equivalently as two inequality constraints $h(x)\leq0,\ -h(x)\leq 0$ without affecting MSCQ.

In the case where $\Gamma$ is convex, MSCQ is proposed in  \cite{YeYe} as a constraint qualification for the M-stationary condition.  Two types of sufficient conditions were given for MSCQ. One is the case when all involved functions are affine and the other is when metric regularity holds.  In this section
by making use of FOSCMS for the split system  in Theorem \ref{ThSuffCondMSSplit}, we derive some new sufficient condition for MSCQ for the constraint system (\ref{Omega}).
Applying the new constraint qualification to  {the} problem in Example \ref{Ex1}, we show that in contrast to the MPCC reformulation under which even the weakest constraint qualification MPEC-GCQ fails at $(\xb, \yb, \lambda)$ for all multipliers $\lambda$, the MSCQ holds at $(\xb, \yb)$ for the original formulation.

In order to apply FOSCMS in  Theorem \ref{ThSuffCondMSSplit}, we need to calculate the linearized cone $T_\Omega^{\rm lin} (\bar z)$ and consequently we need to calculate the tangent cone $T_{{\rm gph}\widehat{N}_\Gamma}(\bar y, -\phi(\bar x, \bar y))$. We now perform this task. First we introduce some notations.
Given  vectors $y\in\Gamma$, $y^\ast\in\R^m$, consider the {\em set of multipliers}
\begin{eqnarray}\label{Lambda}
\Lambda(y,y^\ast):=\big\{\lambda \in\R^q_+\big|\;\nabla g(y)^T\lambda=y^\ast, \lambda_i=0, i\not\in \I(y)
\big\}.
\end{eqnarray}
{For a multiplier $\lambda$, the corresponding collection of {\em strict complementarity indexes} is denoted by}
\begin{eqnarray}\label{EqI+}
I^+(\lambda):=\big\{i\in \{1,\ldots,q\}\big|\;\lambda_i>0\big\}\;\mbox{ for }\;\lambda=(\lambda_1,\ldots,\lambda_q)\in\R^q_+.
\end{eqnarray}
Denote by $\E(y,y^\ast)$ the collection of all the {\em extreme points} of the closed and convex set of multipliers $\Lambda(y,y^\ast)$ and recall that $\lambda\in\Lambda(y,y^\ast)$ belongs to $\E(y,y^\ast)$ if and only if the family of gradients $\{\nabla g_i(y)\mv i\in I^+(\lambda)\}$ is linearly independent. Further $\E(y,y^\ast)\ne\emptyset$ if and only if $\Lambda(y,y^\ast)\ne\emptyset$.
To proceed further, recall the notion of the {\em critical cone} to $\Gamma$ at $(y,y^\ast)\in\Gr
\widehat{N}_\Gamma$ given by
$K(y,y^\ast):=T_\Gamma(y)\cap\{y^\ast\}^\perp$
and define the {\em multiplier set in a direction} $v\in K(y,y^\ast)$ by
\begin{eqnarray}\label{EqDir-mult}
\Lambda(y,y^\ast;v):=\argmax\big\{v^T\nabla^2(\lambda^Tg)(y)v\mv\lambda\in\Lambda(y,y^\ast)\big\}.
\end{eqnarray}
Note that $\Lambda(y,y^\ast;v)$ is the solution set of a linear optimization problem and therefore $\Lambda(y,y^\ast;v)\cap \E(y,y^\ast)\not=\emptyset$ whenever $\Lambda(y,y^\ast;v)\not=\emptyset$. Further we denote the corresponding optimal function value by
\begin{eqnarray}\label{EqDir-multval}
\theta(y,y^\ast;v):=\max\big\{v^T\nabla^2(\lambda^Tg)(y)v\mv\lambda\in\Lambda(y,y^\ast)\big\}.
\end{eqnarray}
%
The critical cone to $\Gamma$ has the following two expressions.
\begin{proposition}(see e.g. \cite[Proposition 4.3]{GfrMo15a})\label{criticalcone}
Suppose that MSCQ holds for the {system} $g(y) \in \R_-^q$ at $y$. Then  the critical cone to $\Gamma$ at $(y,y^\ast)\in\Gr\widehat N_\Gamma$ is a convex polyhedron that  can be explicitly expressed as
$$K(y,y^*)=\{v| \nabla g(y)v \in T_{\R^q_-}(g(y)), v^T y^*=0\}.$$
Moreover for any $\lambda \in \Lambda(y,y^\ast)$,
$$K(y,y^*)=\left \{v| \nabla g(y)v \left \{\begin{array}{ll}
=0 \mbox{ if } \lambda_i>0\\
\leq 0 \mbox{ if } \lambda_i=0 \end{array} \right. \right \}.$$
\end{proposition}
Based on the expression for the critical cone, it is easy to see that the normal cone to the critical cone has the following expression.
\begin{lemma}\cite[Lemma 1]{{GfrOut14}} \label{lemma1} Assume MSCQ holds at $y$ for
the system $g(y)\in\R^q_-$. Let $v\in K(y,y^*), \lambda \in \Lambda(y, y^*)$.  Then
$${N}_{K(y,y^*)}(v)=\{\nabla g(y)^T\mu|\mu^T\nabla g(y)v=0, \mu\in %
T_{N_{\R_-^q}(g(y))}(\lambda)\}.$$
\end{lemma}

We are now ready to  calculate the tangent cone to the graph of $\widehat N_\Gamma$. This result will be needed in the sufficient condition for MSCQ and it is also of an independent interest.
The  first equation in the formula \eqref{EqTanConeGrNormalCone} was first shown in \cite[Theorem~1]{GfrOut14} under the extra assumption that the metric regularity holds locally uniformly except for $\yb$, whereas in \cite{ChiHi16} this extra assumption was removed.
\begin{theorem}\label{ThTanConeGrNormalCone}
Given $\yb\in\Gamma$, assume that MSCQ holds  at $\yb$ for
the system $g(y)\in\R^q_-$. Then  there is a real $\kappa>0$ and a neighorhood $V$ of $\bar y$ such that
for any $y\in\Gamma\cap V$  and any  $y^\ast\in\widehat N_\Gamma(y)$
 the tangent cone to the graph of $\widehat N_\Gamma$ at $(y, y^*)$ can be calculated by
\begin{eqnarray}\label{EqTanConeGrNormalCone}
\lefteqn{T_{\Gr \widehat N_\Gamma}(y,y^\ast)}\\
\nonumber
&=&\big\{(v,v^\ast)\in\R^{2m}\big|\;\exists\,\lambda\in\Lambda(y,y^\ast;v)\;\mbox{ with }\;
v^\ast\in\nabla^2(\lambda^Tg)(y)v+N_{K(y,y^\ast)}(v)\big\}\\
\nonumber&=&\big\{(v,v^\ast)\in\R^{2m}\big|\;\exists\,\lambda\in\Lambda(y,y^\ast;v)\cap \kappa\norm{y^\ast} \B_{\R^q}\;\mbox{ with }\;
v^\ast\in\nabla^2(\lambda^Tg)(y)v+N_{K(y,y^\ast)}(v)\big\},
\end{eqnarray}
where the critical cone $K(y,y^\ast)$ and the normal cone $N_{K(y,y^\ast)}(v)$ can be calculated as in Proposition \ref{criticalcone} and Lemma  \ref{lemma1} respectively, and
  the set  $\Gr \widehat N_\Gamma$ is geometrically derivable at $(y, y^*)$.
\end{theorem}
\begin{proof} Since MSCQ holds  at $\yb$ for
the system $g(y)\in\R^q_-$,  we can find an open neighborhood $V$ of $\yb$ and and a real $\kappa>0$ such that
\begin{equation}
\dist{y,\Gamma}\leq \kappa\dist{g(y),\R^q_-}\ \forall y\in V, \label{errorb} \end{equation}
which means that  MSCQ holds at every $y\in\Gamma\cap V$.  Therefore $K(y,y^\ast)$ and  and
$N_{K(y,y^\ast)}(v)$ can be calculated as in Proposition \ref{criticalcone} and Lemma  \ref{lemma1} respectively. By the proof of the first part of \cite[Theorem~1]{GfrOut14} we obtain that for every $y^\ast\in\widehat N_\Gamma(y)$,
\begin{eqnarray*}
\nonumber\lefteqn{\big\{(v,v^\ast)\in\R^{2m}\big|\;\exists\,\lambda\in\Lambda(y,y^\ast;v)\cap \kappa\norm{y^\ast}	 \B_{\R^q}\;\mbox{ with }\;
v^\ast\in\nabla^2(\lambda^Tg)(y)v+ N_{K(y,y^\ast)}(v)\big\}}\\
\label{EqInclAux1}&\subset &\big\{(v,v^\ast)\in\R^{2m}\big|\;\exists\,\lambda\in\Lambda(y,y^\ast;v)\;\mbox{ with }\;
v^\ast\in\nabla^2(\lambda^Tg)(y)v+ N_{K(y,y^\ast)}(v)\big\}\qquad\\
\label{EqInclAux2}& \subset & \big\{(v,v^\ast)\in\R^{2m}\big| \lim_{t\downarrow 0}t^{-1}\distb{(y+tv,y^\ast+tv^\ast), \Gr \widehat N_\Gamma}=0\}\\
\label{EqInclAux3}&\subset & T_{\Gr \widehat N_\Gamma}(y,y^\ast).
\end{eqnarray*}

We now show  the reversed inclusion
\begin{eqnarray}
\lefteqn{T_{\Gr \widehat N_\Gamma}(y,y^\ast)}\label{reverse}\\
&\subset&\big\{(v,v^\ast)\in\R^{2m}\big|\;\exists\,\lambda\in\Lambda(y,y^\ast;v)\cap \kappa\norm{y^\ast} \B_{\R^q}\;\mbox{ with }\;
v^\ast\in\nabla^2(\lambda^Tg)(y)v+N_{K(y,y^\ast)}(v)\big\}.\nonumber
\end{eqnarray}
 Although the proof technique is essentially the same as \cite[Theorem~1]{GfrOut14}, for completeness we provide the detailed proof.
%
Consider $y\in \Gamma\cap V$, $y^\ast\in  \widehat N_\Gamma(y)$ and
 let $(v,v^*)\in T_{\Gr \widehat N_\Gamma}(y,y^\ast)$. Then by definition of the tangent cone, there exist sequences $t_k\downarrow 0, v_k\rightarrow v, v_k^* \rightarrow v^*$ such that
 $ y_k^*:={y^*}+t_kv_k^* \in \widehat N_\Gamma(y_k)$, where $y_k:=y+t_kv_k$. By passing to a subsequence if necessary we can assume that $y_k\in V$ $\forall k$ and that there is some index set $\widetilde{\I}\subset \I(y)$ such that $\I(y_k)=\widetilde{\I}$ hold for all $k$. For every $i\in {\I(y)}$ we have
 \begin{equation}\label{tylor}
 g_i(y_k)=g_i(y)+t_k\nabla g_i(y)v_k +o(t_k)=t_k\nabla g_i(y)v_k +o(t_k)\left \{\begin{array}{ll}
 =0 &\mbox{ if } i\in \widetilde{\I},\\
 \leq 0 & \mbox{ if } i\in \I(y) \setminus  \widetilde{\I}.\end{array}
\right.
\end{equation}
 Dividing by $t_k$ and passing to the limit we obtain
 \begin{equation}\label{(21)}
 \nabla g_i(y)v \left \{\begin{array}{ll}
 =0 &\mbox{ if } i\in \widetilde{\I},\\
 \leq 0 & \mbox{ if } i\in \I(y) \setminus  \widetilde{\I},\end{array}
\right.
 \end{equation}
 which means  $v\in T_\Gamma^{\rm lin}(y)$.  Since MSCQ holds at every $y\in\Gamma\cap V$, we have that the GACQ holds at $y$ as well and hence $v\in T_\Gamma(y)$.

{Since (\ref{errorb}) holds and $y_k\in V$, $y_k^*\in \widehat N_\Gamma(y_k)=N_\Gamma(y_k)$,} by Theorem \ref{ThBMP}
 there exists a sequence of multipliers $\lambda^k\in\Lambda(y_k,y_k^\ast)\cap\kappa\norm{y_k^\ast} \B_{\R^q}$ as $k\in\N$. Consequently we {assume} that there exists $c_1\geq 0$ such that $\|\lambda^k\|\leq c_1$ for all $k$.
 Let
 \begin{equation}\label{Psi}
 \Psi_{\widetilde{\I}}(y^*):=\{\lambda\in \R^q|\nabla g(y)^T\lambda =y^*, \lambda_i \geq 0, i \in \widetilde{\I}, \lambda_i=0, i \not \in \widetilde{\I}\}.
 \end{equation}
 By Hoffman's Lemma there is some constant $\beta$ such that for  every $y^*\in \R^m$ with
 $\Psi_{\widetilde{\I}}(y^*)\not =\emptyset$ one has
 \begin{equation}\dist{\lambda,  \Psi_{\widetilde{\I}}(y^*)}\leq \beta(\|\nabla g(y)^T\lambda -y^*\|+\sum_{i\in \widetilde{\I}}{\max \{-\lambda_i}, 0 \}+\sum_{i\not \in  \widetilde{\I}}|\lambda_i|) \quad \forall \lambda\in \R^q.\label{hoffman}
 \end{equation}
 Since
 $$\nabla g(y)^T\lambda^k-y^*=t_kv_k^*+(\nabla g(y)-\nabla g(y_k))^T \lambda^k$$
 and $\|\nabla g(y)-\nabla g(y_k)\|\leq c_2\|y_k-y\|=c_2 t_k \|v_k\|$ for some $c_2\geq 0$, by (\ref{hoffman}) we can find for each $k$ some
 $\widetilde{\lambda}^k \in \Psi_{\widetilde{\I}}(y^*)\subset \Lambda (y,y^*)$ with
 $\| \widetilde{\lambda}^k-\lambda^k\| \leq \beta t_k(\|v_k^*{\|}+c_1c_2\|v_k\|)$.
 Taking $\mu^k:=(\lambda^k-\widetilde{\lambda}^k)/t_k$ we have that $(\mu^k)$ is uniformly bounded. By passing to subsequence if necessary we assume
that $(\lambda^k)$ and $(\mu^k)$ are convergent to some $\lambda\in  \Lambda(y,y^*)\cap \kappa \norm{y^\ast} \B_{\R^q}$,
 and some $\mu$ respectively. Obviously the sequence $(\tilde \lambda^k)$ converges to $\lambda$ as well.
 Since $\lambda_i^k=\widetilde{\lambda}^k_i=0, i \not \in   \widetilde{\I}$, {by virtue of (\ref{(21)})} we have ${\mu^k}^T \nabla g(y) v=0 \ \forall k$ implying
 \begin{equation} \label{mu}
 \mu\in (\nabla g(y) v)^\perp .
 \end{equation}
 Taking into account
{${\lambda^k}^T g(y_k)=0$} and (\ref{tylor}), we obtain
 $$0=\lim_{k\rightarrow \infty} \frac{{{\lambda^k}^T }g(y_k)}{t_k}=\lim_{k\rightarrow \infty}{{\lambda^k}^T} \nabla g(y)v_k ={y^*}^Tv. $$ Therefore combining the above with $v\in T_\Gamma(y)$ we have
  \begin{equation}\label{criticalc}
  v\in K(y, y^*).
  \end{equation}

 Further we have for all $\lambda' \in  \Lambda(y,y^\ast)$, since $\widetilde{\lambda}^k\in \Lambda(y,y^\ast)$,
 \begin{eqnarray*}
 0 & \leq & ({\widetilde{\lambda}^k}-\lambda')^T g(y_k) =({\widetilde{\lambda}^k}-\lambda')^T( g(y)+t_k\nabla g(y) v_k+\frac{1}{2} t_k^2 v_k^T \nabla^2 g(y) v_k +o(t_k^2))\\
 &=& ({\widetilde{\lambda}^k}-\lambda')^T (\frac{1}{2} t_k^2 v_k^T \nabla^2 g(y) v_k +o(t_k^2)).
 \end{eqnarray*}
 Dividing by $t_k^2$ and passing to the limit we obtain
 $(\lambda-\lambda')^T v^T \nabla^2 g(y) v\geq 0 \quad \forall \lambda' \in \Lambda(y,y^\ast)$
 and hence $\lambda\in \Lambda(y,y^*;v)$.

 Since $$y_k^*=\nabla g(y)^T \widetilde \lambda^k+t_k v_k^*=\nabla g(y_k)^T \lambda^k,$$
 we obtain
 \begin{eqnarray*}
 v^*&=& \lim_{k\rightarrow \infty} v_k^* =\lim_{k\rightarrow \infty}\frac{\nabla g(y_k)^T \lambda^k-\nabla g(y)^T \widetilde \lambda^k}{t_k}\\
 &=&\lim_{k\rightarrow \infty}\frac{(\nabla g(y_k)-\nabla g(y))^T {\lambda}^k+\nabla g({y})^T (\lambda^k-\widetilde \lambda^k)}{t_k}\\
 &=& \nabla^2 ({{\lambda}}^Tg)(y)v+\nabla g(y)^T \mu.
 \end{eqnarray*}
 If $\mu \in T_{{N}_{\R_-^q}(g(y))}(\lambda)$,
  since (\ref{mu}) holds,  by using Lemma \ref{lemma1} we have $\nabla g(y)^T \mu\in {{N}}_{K(y,y^*)}(v)$ and hence the inclusion (\ref{reverse})
 is proved.
 Otherwise, by taking into account
\[T_{{N}_{\R_-^q}(g(y))}(\lambda)=\{\mu\in \mathbb{R} ^q\mv \mu_i\geq 0 \mbox{ if }\lambda_i=0\}\] and $\mu_i=0$, $i\not\in \widetilde I$,
 the set $J:=\{ i\in \widetilde\I \mv \lambda_i=0, \mu_i <0\}$ is not empty.
Since $\mu^k$ converges to $\mu$, we can choose some index $\bar{k}$ such that
 $\mu^{\bar{k}}_i =(\lambda_i^{\bar{k}}-\widetilde\lambda_i^{\bar{k}})/ t_{\bar{k}}\leq \mu_i/2 \ \forall i \in J$. Set $\widetilde{\mu}:=\mu+2(
\widetilde \lambda^{\bar{k}}-\lambda)/t_{\bar{k}}$.
 Then for all $i$ with $\lambda_i=0$ we have $\widetilde{\mu}_i\geq \mu_i$ and for all $i\in J$ we have
 $$ \widetilde{\mu}_i=\mu_i +2(
 \widetilde\lambda_i^{\bar{k}}-\lambda_i)/t_{\bar{k}}\geq \mu_i +2(
 \widetilde\lambda_i^{\bar{k}}-\tilde\lambda_i^{\bar k})/t_{\bar{k}}\geq 0$$
 and therefore $\widetilde{\mu} \in T_{N_{\R_-^q(g(y))}} (\lambda)$.
Observing that $\nabla g(y)^T\widetilde{\mu}=\nabla g(y)^T{\mu}$ because of $\lambda, \widetilde\lambda^{\bar k}\in \Lambda (y, y^*)$ and taking into account Lemma  \ref{lemma1} we have $\nabla g(y)^T \widetilde{\mu}\in {N}_{K(y,y^*)}(v)$ and hence the inclusion (\ref{reverse})
 is proved. This finishes the proof of the theorem.

\end{proof}

Since the regular normal cone is the polar of the tangent cone, the following characterization of the regular normal cone of $\Gr\widehat N_\Gamma$ follows from the formula for the tangent cone in Theorem \ref{ThTanConeGrNormalCone}.
\begin{corollary}\label{CorSecOrd} Assume that MSCQ is satisfied for the system $g(y)\leq0$ at $\yb\in\Gamma$. Then there is a neighborhood $V$ of $\yb$ such that for every $(y,y^\ast)\in\Gr\widehat N_\Gamma$ with $y\in V$ the following assertion holds: given any pair $(w^\ast,w)\in \widehat N_{\Gr\widehat N_\Gamma}(y,y^\ast)$ we have $w\in K(y,y^\ast)$ and
\begin{equation}\label{EqBasicIneqTiltStab}
\skalp{w^\ast,w}+ w^T\nabla^2(\lambda^Tg)(y)w \leq 0\;\mbox{ whenever }\;\lambda\in\Lambda(y,y^\ast;w).
\end{equation}
\end{corollary}
\begin{proof}
Choose $V$ such that \eqref{EqTanConeGrNormalCone} holds true for every $y\in \Gamma\cap V$ and consider any $(y,y^\ast)\in\Gr\widehat N_\Gamma$ with $y\in V$ and $(w^\ast,w)\in \widehat N_{\Gr\widehat N_\Gamma}(y,y^\ast)$. By the definition of the regular normal cone we have $\widehat N_{\Gr\widehat N_\Gamma}(y,y^\ast)=\big( T_{\Gr\widehat N_\Gamma}(y,y^\ast)\big)^\circ$ and, since $\{0\}\times N_{K(y,y^\ast)}(0)\subset T_{\Gr\widehat N_\Gamma}(y,y^\ast)$, we obtain
\[\skalp{w^\ast,0}+\skalp{w,v^\ast}\leq 0 \ \forall v^\ast \in N_{K(y,y^\ast)}(0)=K(y,y^\ast)^\circ,\]
implying $w\in \cl\co K(y,y^\ast)=K(y,y^\ast)$. By \eqref{EqTanConeGrNormalCone} we have  $(w, \nabla^2(\lambda^Tg)(y)w)\in T_{\Gr\widehat N_\Gamma}(y,y^\ast)$ for every $\lambda\in\Lambda(y,y^\ast;w)$ and therefore the claimed inequality
\[\skalp{w^\ast,w}+\skalp{w, \nabla^2(\lambda^Tg)(y)w}=\skalp{w^\ast,w}+ w^T\nabla^2(\lambda^Tg)(y)w \leq 0\]
follows.
\end{proof}

The following result  will be needed in the proof of Theorem \ref{ThSuffCondMS_FOGE}.
\begin{lemma}\label{LemBndSecOrdMult}Given $\yb\in\Gamma$, assume that MSCQ holds at $\yb$. Then there is a real $\kappa'>0$ such that for any $y\in\Gamma$ sufficiently close to $\yb$, any normal  vector $y^\ast\in\widehat N_\Gamma(y)$ and any critical direction $v\in K(y,y^\ast)$ one has
\begin{equation}\label{EqBndSecOrdMult}\Lambda(y,y^\ast;v)\cap\E(y,y^\ast)\cap \kappa'\norm{y^\ast} \B_{\R^q}\not=\emptyset.\end{equation}
\end{lemma}
\begin{proof}
  Let $\kappa>0$ be chosen according to Theorem \ref{ThTanConeGrNormalCone} and consider $y\in\Gamma$ as close to $\yb$ such that MSCQ holds at $y$ and \eqref{EqTanConeGrNormalCone} is valid for every $y^\ast\in \widehat N_\Gamma(y)$. Consider $y^\ast\in  \widehat N_\Gamma(y)$ and a critical direction $v\in K(y,y^\ast)$. By \cite[Proposition 4.3]{GfrMo15a} we have $\Lambda(y,y^\ast;v)\not=\emptyset$ and, by taking any $\lambda\in \Lambda(y,y^\ast;v)$, we obtain from Theorem \ref{ThTanConeGrNormalCone} that $(v,v^\ast)\in T_{\Gr \widehat N_\Gamma}(y,y^\ast)$ with $v^\ast=\nabla^2(\lambda^Tg)(y)v$. Applying Theorem \ref{ThTanConeGrNormalCone} once more, we see that $v^\ast\in \nabla^2(\tilde\lambda^Tg)(y)v+N_{K(y,y^\ast)}(v)$ with $\tilde\lambda\in\Lambda(y,y^\ast;v)\cap \kappa\norm{y^\ast} \B_{\R^q}$ showing that $\Lambda(y,y^\ast;v)\cap \kappa\norm{y^\ast}\B_{\R^q}\not=\emptyset$. Next consider a solution $\bar\lambda$ of the linear optimization problem
  \[\min\sum_{i=1}^q\lambda_i\quad\mbox{subject to }\lambda\in \Lambda(y,y^\ast;v).\]
  We can choose $\bar\lambda$ as an extreme point of the polyhedron $\Lambda(y,y^\ast;v)$ implying $\bar\lambda\in \E(y,y^\ast)$.
  Since $\Lambda(y,y^\ast;v)\subset \R^q_+$, we obtain
  $$\norm{\bar\lambda}\leq \sum_{i=1}^q\vert\bar\lambda_i\vert=\sum_{i=1}^q \bar \lambda_i\leq \sum_{i=1}^q\tilde\lambda_i\leq \sqrt{q}\norm{\tilde\lambda}\leq\sqrt{q}\kappa\norm{y^\ast},$$ and hence \eqref{EqBndSecOrdMult} follows with $\kappa'=\kappa\sqrt{q}$.
\end{proof}

We are now in {position} to state a verifiable sufficient condition for MSCQ to hold for problem (MPEC).
\begin{theorem}\label{ThSuffCondMS_FOGE}
  Given $(\xb,\yb)\in \Omega$ defined as in (\ref{Omega}), assume that  MSCQ holds both  for the lower level problem constraints $g(y)\leq 0$ at $\yb$ and for the upper  level constraints $G(x,y)\leq0$ at
  $(\xb,\yb)$. Further assume that
  \begin{equation}\label{EqNonDegG}
  \nabla_x G(\xb,\yb)^T\eta=0,\ \eta\in N_{\R^p_-}(G(\xb,\yb))\quad \Longrightarrow\quad \nabla_y G(\xb,\yb)^T\eta=0
  \end{equation}
  and assume that
  there do not exist $(u,v)\not=0$, $\lambda\in\Lambda(\yb,-\phi(\xb,\yb);v)\cap \E(\yb,-\phi(\xb,\yb))$, $\eta\in\R^p_+$ and $w\not=0$ satisfying
  \begin{eqnarray}
    \label{EqSuffMS1}&&\nabla G(\xb,\yb)
   (u,v)\in T_{\R^p_-}(G(\xb,\yb)),\; \\
    \label{EqSuffMS2}
    &&(v,-\nabla_{x}\phi(\xb,\yb)u-\nabla_{y}\phi(\xb,\yb)v)\in T_{\Gr \widehat N_\Gamma}(\yb,-\phi(\xb,\yb)),\\
    \label{EqSuffMS3}&&-\nabla_{x}\phi(\xb,\yb)^Tw+\nabla_{x} G(\xb,\yb)^T\eta=0,\;\eta\in N_{\R^p_-}(G(\xb,\yb)),\; \eta^T\nabla G(\xb,\yb)(u,v)=0,\\
    \label{EqSuffMS4}&&\nabla g_i(\yb)w=0, i\in I^+(\lambda),\; w^T\left (\nabla_{y}\phi(\xb,\yb)+\nabla^2(\lambda^Tg(\yb)\right )w-\eta^T\nabla_y G(\xb,\yb)w\leq 0,\qquad
  \end{eqnarray}
  where the tangent cone $ T_{\Gr \widehat N_\Gamma}(\yb,-\phi(\xb,\yb))$ can be calculated as in Theorem \ref{ThTanConeGrNormalCone}.
Then the multifunction $M_{\rm MPEC}$ defined by
\begin{equation}\label{MPEC}
M_{\rm MPEC}{(x,y)}:=\vek{\phi(x,y)+\widehat N_\Gamma(y)\\G(x,y)-\R^p_-}
\end{equation} is metrically subregular at $\big((\xb,\yb),0\big)$.
\end{theorem}
\begin{proof}By 
Proposition \ref{PropEquGrSubReg},  it suffices to show that the multifunction  $P(x,y)-D$ with $P$ and $D$ given by
$$ P(x,y):=\left(\begin{array}{c}y,-\phi(x,y)\\G(x,y)\end{array}\right) \mbox{ and } D:=\Gr \widehat N_\Gamma\times\R^p_-$$  is metrically subregular at $\big((\xb,\yb),0\big)$.
 We now  invoke Theorem \ref{ThSuffCondMSSplit} with
 $$P_1(x,y):=(y,-\phi(x,y)),\ P_2(x,y):=G(x,y),\ D_1:=\Gr \widehat N_\Gamma,D_2:=\R^p_-.$$ By the assumption $P_2(x,y)-D_2$ is metrically subregular at {$\big((\xb,\yb),0\big)$}.  Assume to the contrary that $P(\cdot,\cdot)-D$ is not metrically subregular at $\big((\xb,\yb),0\big)$. Then by Theorem \ref{ThSuffCondMSSplit}, there exist $0\not=z=(u,v) \in T^{\rm lin}_\Omega(\xb,\yb)$ and a directional limiting normal $z^\ast=(w^\ast,w,\eta)\in\R^m\times\R^m\times\R^p$  such that $\nabla P(\xb,\yb)^Tz^\ast=0$, $(w^\ast,w)\in N_{\Gr \widehat N_\Gamma}(P_1(\xb,\yb); \nabla P_1(\xb,\yb)z)$, $\eta\in N_{\R^p_-}\big(G(\xb,\yb);\nabla G(\xb,\yb)(u,v)\big)$ and $(w^\ast,w)\not=0$.

  Hence
  \begin{equation}\label{w*eqn}
  0=\nabla P(\xb,\yb)^Tz^\ast=\left(\begin{array}{c}-\nabla_{x}\phi(\xb,\yb)^Tw+\nabla_{x}G(\xb,\yb)^T\eta\\
  w^\ast-\nabla_{y}\phi(\xb,\yb)^Tw+\nabla_y G(\xb,\yb)^T\eta
  \end{array}\right).\end{equation}
Since $z=(u,v) \in T^{\rm lin}_\Omega(\xb,\yb)$, by  the rule of tangents to product sets from Proposition \ref{productset} we obtain
  \[\nabla P(\xb,\yb)z=\left(\begin{array}{c}(v,-\nabla_{x}\phi(\xb,\yb)u-\nabla_{y}\phi(\xb,\yb)v)\\
  \nabla G(\xb,\yb)(u,v)\end{array}\right)\in T_{\Gr \widehat N_\Gamma}(\yb,\yba)\times T_{\R^p_-}\big(G(\xb,\yb)\big),\]
  where $\yba:=-\phi(\xb,\yb)$.
  It follows that $
  (v,-\nabla_{x}\phi(\xb,\yb)u-\nabla_{y}\phi(\xb,\yb)v)\in T_{\Gr \widehat N_\Gamma}(\yb,\yba)$ and consequently by Theorem \ref{ThTanConeGrNormalCone} we have  $v\in K(\yb,\yba)$.
  Further we deduce from Proposition \ref{relationship}  that
  \[\eta\in N_{\R^p_-}(G(\xb,\yb)), \ \eta^T\nabla G(\xb,\yb)(u,v)=0.\]
  So far we have shown  that $u,v,\eta,w$ fulfill \eqref{EqSuffMS1}-\eqref{EqSuffMS3}. Further we have $w\not=0$, because if $w=0$ then by virtue of \eqref{EqNonDegG} and \eqref{w*eqn} we would obtain $\nabla_xG(\xb,\yb)^T\eta=0$, $\nabla_yG(\xb,\yb)^T\eta=0$ and consequently $w^\ast=0$ contradicting $(w^\ast,w)\not=0$. If  we can show the existence of $\lambda\in\Lambda(\yb,\yba;v)\cap \E(\yb,\yba)$ such that \eqref{EqSuffMS4} holds,  then
  we have obtained the desired contradiction to our assumptions, and this would complete the proof.

Since $(w^\ast,w)\in N_{\Gr \widehat N_\Gamma}(P_1(\xb,\yb); \nabla P_1(\xb,\yb)z)$,   by
the  definition of the directional limiting normal cone, there are sequences $t_k\downarrow 0$, $d_k=(v_k,v_k^\ast)\in\R^m\times\R^m$ and $(w_k^\ast,w_k)\in\R^m\times\R^m$  satisfying $(w_k^\ast,w_k)\in\widehat N_{\Gr \widehat N_\Gamma}(P_1(\xb,\yb)+t_kd_k)$ $\forall k$ and $(d_k,w_k^\ast,w_k)\to(\nabla P_1(\xb,\yb)z,w^\ast,w)$. That is,  $(y_k,y_k^\ast):=(\yb,\yba)+t_k(v_k,v_k^\ast)\in \Gr \widehat N_\Gamma$, 
 $(w_k^\ast,w_k)\in \widehat N_{\Gr \widehat N_\Gamma}(y_k,y_k^\ast)$
and $(v_k, v_k^\ast)\to (v, -\nabla_{x}\phi(\xb,\yb)u-\nabla_{y}\phi(\xb,\yb)v)$.
  By passing to a subsequence if necessary, we can assume that MSCQ holds for $g(y)\leq 0$ at $y_k$ for all $k$ and by invoking Corollary \ref{CorSecOrd} we obtain $w_k\in K(y_k,y_k^\ast)$, and
  \begin{equation}\label{EqSecOrdAux1}{w_k^\ast}^Tw_k+w_k^T\nabla^2(\lambda^T g)(y_k)w_k\leq 0 \mbox{ whenever }\lambda\in\Lambda(y_k,y_k^\ast;w_k).
  \end{equation}
  By Lemma \ref{LemBndSecOrdMult} we can find a uniformly bounded sequence $\lambda^k\in\Lambda(y_k,y_k^\ast;w_k)\cap \E(y_k,y_k^\ast)$. In particular, following from the proof of Lemma \ref{LemBndSecOrdMult}, we can choose $\lambda^k$ as an optimal solution of the linear optimization problem
  \begin{equation}\label{EqMinL1Norm}\min\sum_{i=1}^q\lambda_i\;\mbox{ subject to }\;\lambda\in \Lambda(y_k,y_k^\ast;w_k).
  \end{equation}
  By passing once more to a subsequence if necessary, we can assume that $\lambda^k$ converges to $\lb$, and we easily conclude $\lb\in\Lambda(\yb,\yba)$ and
  ${w^\ast}^Tw+w^T\nabla^2(\lb^T g)(\yb)w\leq 0$, which  together with $w^\ast-\nabla_{y}\phi(\xb,\yb)^Tw+\nabla_y G(\xb,\yb)^T\eta=0$ (see (\ref{w*eqn})) results in
  \begin{equation}
    \label{EqSecOrdAux2}w^T\left (\nabla_{y}\phi(\xb,\yb)+\nabla^2(\bar\lambda^Tg)(\yb)\right )w-\eta^T\nabla_y G(\xb,\yb)w\leq 0.
  \end{equation}
  Further, we can assume that $I^+(\lb)\subset I^+(\lambda^k)$ and therefore, because of $\lambda^k\in N_{\R^q_-}(g(y_k))$, $\lb^Tg(y_k)={\lambda^k}^Tg(y_k)=0$. Hence for every $\lambda\in\Lambda(\yb,\yba)$ we obtain
  \begin{eqnarray*}
  0&\geq& (\lambda-\lb)^Tg(y_k)\\
  &=&(\lambda-\lb)^Tg(\yb)+\nabla((\lambda-\lb)^Tg)(\yb)(y_k-\yb)\\
  &&\quad+\frac 12 (y_k-\yb)^T\nabla^2((\lambda-\lb)^Tg)(\yb)(y_k-\yb) +\oo(\norm{y_k-\yb}^2)\\
  &=&\frac{t_k^2}2 v_k^T\nabla^2((\lambda-\lb)^Tg)(\yb)v_k+\oo(t_k^2\norm{v_k}^2).
  \end{eqnarray*}
  Dividing by $t_k^2/2$ and passing to the limit yields $0\geq v^T\nabla^2((\lambda-\lb)^Tg)(\yb)v$ and thus $\lb\in\Lambda(\yb,\yba;v)$. Since   $w_k\in K(y_k,y_k^\ast)$ by Proposition \ref{criticalcone} we have $\nabla g_i(y_k)w_k=0$, $i\in I^+(\lambda^k)$  from which $\nabla g_i(\yb)w=0$, $i\in I^+(\lb)$ follows.

 It is known that the  polyhedron $\Lambda( \yb, \yb^*)$ can be represented as the sum  of the convex hull of its extreme points $ \E(\yb,\yba)$ and its recession cone ${\cal R}:=\{\lambda  \in N_{\R^q_-}(g(\yb))| \nabla g(\yb)^T\lambda=0 \}$.  We  show by contradiction that $\lb\in\co \E(\yb,\yba)$. Assuming on the contrary that $\lb\not\in \co \E(\yb,\yba)$, then $\lb$ has the representation $\lb=\lambda^c+\lambda^r$ with $\lambda^c\in \co \E(\yb,\yba)$ and $\lambda^r\not=0$ belongs to the recession cone ${\cal R}$, i.e.
  \begin{equation}\label{recessioncone}
  \lambda^r\in N_{\R^q_-}(g(\yb)),\ \nabla g(\yb)^T\lambda^r=0.\end{equation}
  Since $\lambda^k\in\Lambda(y_k,y_k^\ast;w_k)$, it is a solution to the linear program:
  \begin{eqnarray*}
  \max_{\lambda\geq 0} &&  w_k^T\nabla^2(\lambda^Tg)(y_k)(w_k)\\
  s.t. && \nabla g(y_k)^T\lambda=y_k^*\\
  && \lambda^Tg(y_k)=0.
  \end{eqnarray*}
   By duality theory of linear programming, for each $k$ there is some $r_k\in\R^m$ verifying
  \[\nabla g_i(y_k)r_k+w_k^T\nabla^2 g_i(y_k)w_k\leq 0,\ \lambda^k_i(\nabla g_i(y_k)r_k+w_k^T\nabla^2 g_i(y_k)w_k)=0,\ i\in \I(y_k).\]
  Since $\Lambda(y_k,y_k^\ast;w_k)=\{\lambda\in\Lambda(y_k,y_k^\ast)\mv w_k^T\nabla^2(\lambda^Tg)({y_k})w_k\geq\theta(y_k,y_k^\ast;w_k)\}$ and $\lambda^k$ solves \eqref{EqMinL1Norm}, again by duality theory of linear programming we can find for each $k$ some $s_k\in\R^m$ and $\beta_k\in\R_+$ such that
  \[\nabla g_i(y_k)s_k+\beta_k w_k^T\nabla^2 g_i(y_k)w_k\leq 1,\ \lambda^k_i(\nabla g_i(y_k)s_k+\beta_k w_k^T\nabla^2 g_i(y_k)w_k-1)=0,\ i\in \I(y_k).\]
  Next we define for every $k$ the elements $\tilde \lambda^k\in\R^q_+$, $\xi_k^\ast\in\R^m$ by
  \begin{eqnarray}
  && \tilde\lambda^k_i:=\begin{cases}
    \lambda^r_i&\mbox{if $i\in I^+(\lambda^r)$,}\\
    \frac 1k&\mbox{if $i\in I^+(\lambda^k)\setminus I^+(\lambda^r)$,}\\
    0&\mbox{else,}\quad
  \end{cases} \nonumber\\
&&  \xi_k^\ast:=\nabla g(y_k)^T\tilde\lambda^k.\label{xi}
\end{eqnarray}
  Since $I^+(\lambda^r)\subset I^+(\lb)\subset I^+(\lambda^k)$, we obtain $I^+(\tilde\lambda^k)=I^+(\lambda^k)$, $\tilde\lambda^k\in N_{\R^q_-}(g(y_k))$  and $\xi_k^\ast\in N_\Gamma(y_k)$.
   Thus $w_k\in K(y_k,\xi_k^\ast)$ by Proposition \ref{criticalcone} and
  \[\nabla g_i(y_k)r_k+w_k^T\nabla^2 g_i(y_k)w_k\leq 0,\ \tilde\lambda^k_i(\nabla g_i(y_k)r_k+w_k^T\nabla^2 g_i(y_k)w_k)=0,\ i\in \I(y_k)\]
  implying $\tilde\lambda^k\in\Lambda(y_k,\xi_k^\ast;w_k)$ by duality theory of linear programming. Moreover, because of $I^+(\tilde\lambda^k)=I^+(\lambda^k)$ we also have
  \[\nabla g_i(y_k)s_k+\beta_k w_k^T\nabla^2 g_i(y_k)w_k\leq 1,\ \tilde\lambda^k_i(\nabla g_i(y_k)s_k+\beta_k w_k^T\nabla^2 g_i(y_k)w_k-1)=0,\ i\in \I(y_k),\]
  implying that $\tilde\lambda^k$ is solution of the linear program
  \[\min\sum_{i=1}^q\lambda_i\;\mbox{ subject to }\;\lambda\in \Lambda(y_k,\xi_k^\ast;w_k),\]
  and, together with $\Lambda(y_k,\xi_k^\ast;w_k)\subset \R^q_+$,
  \[\min\{\norm{\lambda}\mv \lambda\in \Lambda(y_k,\xi_k^\ast;w_k)\}\geq \frac 1{\sqrt{q}}\min\{\sum_{i=1}^q\lambda_i\mv\lambda\in \Lambda(y_k,\xi_k^\ast;w_k)\}{\geq}\frac{\sum_{i=1}^q\lambda_i^r}{\sqrt{q}}:=\beta>0.\]
  Taking into account that $\lim_{k\to\infty}\tilde\lambda^k=\lambda^r$ and (\ref{recessioncone}), (\ref{xi}), we conclude $\lim_{k\to\infty}\norm{\xi_k^\ast}=0$, showing that for every real $\kappa'$ we have
  \[\Lambda(y_k,\xi_k^\ast;w_k)\cap \E(y_k,\xi_k^\ast)\cap\kappa'\norm{\xi_k^\ast} \B_{\R^q}\subset \Lambda(y_k,\xi_k^\ast;w_k)\cap \kappa'\norm{\xi_k^\ast}  \B_{\R^q}=\emptyset\]
  for all $k$ sufficiently large contradicting the statement of Lemma \ref{LemBndSecOrdMult}. Hence $\lb\in \co\E(\yb,\yba)$ and thus $\lb$ admits a representation as convex combination
  \[\lb=\sum_{j=1}^N\alpha_j{\hat\lambda^j}\;\mbox{ with }\; \sum_{j=1}^N \alpha_j=1,\;0<\alpha_j\leq 1,\; {\hat\lambda^j}\in\E(\yb,\yba),\;j=1,\ldots,N.\]
  Since $\lb\in \Lambda(\yb,\yba;v)$ we have $\theta(\yb,\yba;v)=v^T\nabla^2(\lb^Tg)(\yb)v=\sum_{j=1}^N \alpha_j v^T\nabla^2({\hat\lambda{}^j}^Tg)(\yb)v$ implying, together with $v^T\nabla^2({\hat\lambda{}^j}^T g)(\yb)v\leq \theta(\yb,\yba;v)$, that $v^T\nabla^2({{\hat\lambda{}^j}}^Tg)(\yb)v= \theta(\yb,\yba;v)$ and consequently ${\hat\lambda^j}\in \Lambda(\yb,\yba;v)$. It follows from (\ref{EqSecOrdAux2}) that
  \begin{eqnarray*}\lefteqn{\sum_{j=1}^N\alpha_j\left( w^T\big (\nabla_{y}\phi(\xb,\yb)+\nabla^2({\hat\lambda{}^j}^Tg)(\yb)\big )w-\eta^T\nabla_y G(\xb,\yb)w\right)}\\
  &&=w^T\left (\nabla_{y}\phi(\xb,\yb)+\nabla^2(\bar\lambda^Tg)(\yb)\right )w-\eta^T\nabla_y G(\xb,\yb)w\leq 0\end{eqnarray*}
  and hence there exists some index $\bar j$ with
  $$w^T\left (\nabla_{y}\phi(\xb,\yb)+\nabla^2({\hat\lambda{}^{\bar j}}^Tg)(\yb)\right )w-\eta^T\nabla_y G(\xb,\yb)w\leq 0.$$
  Further, by Proposition \ref{criticalcone} we have $\nabla g_i(\yb)w=0$ $\forall i\in I^+(\lb)\supset I^+({\hat\lambda^{\bar j}})$ and we see that \eqref{EqSuffMS4} is fulfilled with $\lambda={\hat\lambda^{\bar j}}$.
\end{proof}

\if{
By virtue of \cite[Theorem 3.2]{YeYe}, we have the following necessary optimality condition for (MPEC).
\begin{proposition}
 Let $(\xb,\yb)\in \Omega$ defined as in (\ref{Omega}) be a local optimal solution of problem (MPEC) where all functions $F, \phi, G$ are continuously differentiable and $g$ is twice continuously differentiable. Suppose all constraint qualifications in Theorem \ref{ThSuffCondMS_FOGE} hold at $(\xb,\yb)$. Then there are multipliers $(\mu, \nu)\in \mathbb{R}^m\times \mathbb{R}^p$ and $\varrho\in \mathbb{R}^p$ which solve the system
 \begin{eqnarray*}
 && 0=\nabla_x F(\bar x, \bar y)-\nabla_x \phi(\bar x,\bar y)^T \nu +\nabla_x G(\bar x,\bar y)^T \varrho,\\
 && 0=\nabla_y F(\bar x, \bar y)-\nabla_y\phi(\bar x,\bar y)^T \nu +\mu+\nabla_y G(\bar x,\bar y)^T \varrho,\\
 && (\mu, \nu) \in N_{gph\widehat{N}_\Gamma}(\bar y, -\phi(\bar x,\bar y)),\\
 &&\varrho\in N_{\mathbb{R}^p_-}(G(\bar x,\bar y)).
 \end{eqnarray*}
\end{proposition}}\fi

\begin{example}[Example \ref{Ex1} revisited]  \label{Ex1GE}
Instead of reformulating the MPEC as a (MPCC), we consider the MPEC in the original form (MPEC).
Since for the constraints $g(y)\leq 0$ of the lower level problem MFCQ is fulfilled at $\yb$ and the gradients of the upper level constraints $G(x,y)\leq 0$ are linearly independent, MSCQ holds for both constraint systems. Condition \eqref{EqNonDegG} is obviously fulfilled due to $\nabla_y G(x,y)=0$. Setting $\yba:=-\phi(\xb,\yb)=(0,0,1)$, as in Example \ref{Ex1} we  obtain
\[ \Lambda(\yb,\yba)=\{(\lambda_1,\lambda_2)\in\R^2_+\mv\lambda_1+\lambda_2=1\}.\]
Since $\nabla g_1(\yb)=\nabla g_2(\yb)=(0,0,1)$ and for every $\lambda\in\Lambda(\yb,\yba)$ either $\lambda_1>0$ or $\lambda_2>0$, we deduce
\[W(\lambda):=\{w\in\R^3\mv \nabla g_i(\yb)w=0,\;i\in I^+(\lambda)\}={\mathbb{R}^2}\times\{0\}\ \  \forall \lambda\in\Lambda(\yb,\yba).\]
Since
\[w^T\left (\nabla_{y}\phi(\xb,\yb)+\nabla^2(\lambda^Tg)(\yb)\right )w-\eta^T\nabla_y G(\xb,\yb)w=(1+\lambda_1)w_1^2+(1+\lambda_2)w_2^2\geq 0\]
 there cannot exist $0 \not=w \in W(\lambda)$  and $\lambda\in\Lambda(\yb,\yba)$ fulfilling \eqref{EqSuffMS4}. Hence by virtue of Theorem \ref{ThSuffCondMS_FOGE}, MSCQ holds at $(\xb, \yb)$.

\if{ we easily obtain
\[K(\yb,\yba)=\R^2\times\{0\}, \  \Lambda(\yb,\yba)=\{(\lambda_1,\lambda_2)\in\R^2_+\mv\lambda_1+\lambda_2=1\}.\]
Since $\nabla g_1(\yb)=\nabla g_2(\yb)=(0,0,1)$ and for every $\lambda\in\Lambda(\yb,\yba)$ either $\lambda_1>0$ or $\lambda_2>0$, we deduce
\[\{w\in\R^3\mv \nabla g_i(\yb)w=0,\;i\in I^+(\lambda)\}=R^2\times\{0\}\ \forall \lambda\in\Lambda(\yb,\yba)\]
Since
\[w^T\left (\nabla_{y}\phi(\xb,\yb)+\nabla^2(\lambda^Tg)(\yb)\right )w-\eta^T\nabla_y G(\xb,\yb)w=(1+\lambda_1)w_1^2+(1+\lambda_2)w_2^2\]
 there cannot exist $w\not=0$}\fi
\end{example}

\section*{Acknowledgements}
The research of the first author was supported by the Austrian Science Fund (FWF) under grants P26132-N25 and P29190-N32. The research of the second author was partially supported by NSERC.
{The authors would like to thank the two anonymous reviewers for their extremely careful review and valuable comments that have helped 
{us to improve} the presentation of the manuscript.}

 \end{document}